\documentclass[reqno]{amsart}
\usepackage{mathrsfs}
\usepackage{amsmath}
\usepackage{amsthm}
\usepackage{amsfonts}
\usepackage{amssymb}
\usepackage{url}
\usepackage{enumerate}
\usepackage[pdftex,bookmarks=true]{hyperref}
\usepackage[usenames, dvipsnames]{color}
\usepackage{verbatim}

\usepackage{pdfsync}

\newcommand\Vol{{\operatorname{Vol}}}
\newcommand\rank{{\operatorname{rank}}}

\newcommand\R{{\mathbb{R}}}

\newcommand\Q{{\mathbf{Q}}}

\renewcommand\P{{\mathbf{P}}}
\newcommand\E{{\mathbf{E}}}

\newcommand\tr{{\operatorname{tr}}}

\newcommand\dist{{\operatorname{dist}}}
\newcommand\Z{{\mathbb{Z}}}

\newcommand\F{{\mathbb{F}}}
\newcommand\I{{\mathbf{I}}}
\newcommand\col{{\mathbf{c}}}
\newcommand\row{{\mathbf{r}}}

\newcommand\al{\alpha}
\newcommand\la{\lambda}

\newcommand\Ba{{\mathbf a}}
\newcommand\Bb{{\mathbf b}}

\newcommand\Be{{\mathbf e}}
\newcommand\Bf{{\mathbf f}}
\newcommand\Bg{{\mathbf g}}

\newcommand\Bn{{\mathbf n}}

\newcommand\Bp{{\mathbf p}}

\newcommand\Br{{\mathbf r}}

\newcommand\Bu{{\mathbf u}}
\newcommand\Bv{{\mathbf v}}
\newcommand\Bw{{\mathbf w}}
\newcommand\Bx{{\mathbf x}}

\newcommand\Bz{{\mathbf z}}

\newcommand\BB{{\mathbf B}}

\newcommand\BF{{\mathbf F}}

\newcommand\CE{{\mathcal E}}

\newcommand\CN{{\mathcal N}}

\newcommand\CS{{\mathcal S}}

\newcommand\CU{{\mathcal U}}
\newcommand\CV{{\mathcal V}}
\newcommand\CW{{\mathcal W}}

\renewcommand\mod{\ \operatorname{mod}\ }

\newcommand\LCD{\mathbf{LCD}}

\newcommand\supp{\operatorname{supp}}

\newcommand\ULCD{\mathbf{ULCD}}

\newcommand\eps{\varepsilon}

\renewcommand\a{\alpha}

\newcommand\Aut{\mathbf{Aut}}

\newcommand\be{\beta}

\newcommand\Mat{\operatorname{Mat}}

\usepackage{fullpage}

\theoremstyle{plain}
\newtheorem{theorem}[subsection]{Theorem}

\newtheorem{prop}[subsection]{Proposition}
\newtheorem{proposition}[subsection]{Proposition}
\newtheorem{fact}[subsection]{Fact}
\newtheorem{lemma}[subsection]{Lemma}
\newtheorem{corollary}[subsection]{Corollary}
\newtheorem{cor}[subsection]{Corollary}

\newtheorem{remark}[subsection]{Remark}

\newtheorem{claim}[subsection]{Claim}

\theoremstyle{definition}
\newtheorem{definition}[subsection]{Definition}

\begin{document}
	\newcommand{\Sean}[1]{{\color{Green} \sf $\clubsuit\clubsuit\clubsuit$ Sean: [#1]}}

	\title{Some new results in random matrices over finite fields}
	

	\author{Kyle Luh}
	
	\address{Center of Mathematical Sciences and Applications\\ Harvard University\\ 20 Garden St.\\ Cambridge, MA 02138 USA }
	\email{kluh@cmsa.fas.harvard.edu}
	
	\author{Sean Meehan}
	\address{Department of Mathematics\\ The Ohio State University \\ 231 W 18th Ave \\ Columbus, OH 43210 USA}
	\email{meehan.73@buckeyemail.osu.edu}

	\author{Hoi H. Nguyen}
	\address{Department of Mathematics\\ The Ohio State University \\ 231 W 18th Ave \\ Columbus, OH 43210 USA}
	\email{nguyen.1261@math.osu.edu}

	
	\begin{abstract} In this note  we give various characterizations of random walks with possibly different steps that have relatively large discrepancy from the uniform distribution modulo a prime $p$, and use these results to study the distribution of the rank of random matrices over $\F_p$ and the equi-distribution behavior of normal vectors of random hyperplanes. We also study the probability that a random square matrix is eigenvalue-free, or when its characteristic polynomial is divisible by a given irreducible polynomial in the limit $n\to \infty$ in $\F_p$. We show that these statistics are universal, extending results of Stong and Neumann-Praeger beyond the uniform model. \end{abstract}

	\maketitle

	\section{Introduction } 
	Let $q$ be a prime power, and let $\Mat(n, q)$ ($GL(n, q)$) be the group of all $n \times n$ (resp. non-singular) matrices with entries in the field of $q$ elements. Given a (class) function $f$ that depends on $M$ (such as the rank of $M$, or the factors of its characteristic polynomial, etc), it is natural to study the behavior of $f$ for a ``typical" matrix, such as for one sampled uniformly at random, we call this the uniform model.  
	
	\subsection{Some statistics for the uniform model} 
	Our first example is on the rank of $M$. By exposing the columns of $M$ one by one, it is not hard to show that the probability that $M$ belongs to $GL(n,q)$ is exactly
	$$\P(M \in GL(n,\F_q)) = \frac{(q^n-1)(q^n-q) \dots (q^n - q^{n-1})}{q^{n^2}} = \prod_{i=1}^n (1 - q^{-i}).$$
	More generally, for $0\le k\le n$ we can show that
	\begin{equation}\label{rank:uniform}
	\P(M \mbox{ has rank } n -k) = \frac{1}{q^{k^2}} \frac{\prod_{i=1}^{n} (1-q^{-i}) \prod_{i=k+1}^n (1-q^{-i})}{ \prod_{i=1}^{k} (1-q^{-i}) \prod_{i=1}^{n-k} (1-q^{-i}) }.
	\end{equation}

	Our next example is the probability that $M$ does not have eigenvalue in $\F_q$ (equivalently, $M$ does not have one dimensional invariant subspace). Beautiful results by Stong \cite{Stong} and Neumann-Praeger \cite{NP1} (see also \cite{Fulman-thesis}) showed that this probability tends to the derangement probability of a random permutation.  We have
	\begin{equation}\label{eqn:free}
	\lim_{q\to \infty} \lim_{n\to \infty} \P(M \mbox{ is eigenvalue-free}) = 1/e.
	\end{equation}
	More generally, Stong showed that in the $q \to \infty$ limit, the probability that the characteristic polynomial of $M$ factors into $n_i$ degree $i$ irreducible factors is the same as the probability of an element of $S_n$ factors into $n_i$ cycles of degree $i$, and Hansen and Schmutz \cite{HSch} also obtained similar results for joint cycle structures. 
	
	In a different direction, random matrices over finite field is also a source to generate random partitions. For instance it follows from \cite{Fulman-BAMS} (and the references therein) for the uniform model:
	\begin{equation}\label{eqn:z-1}
	\lim_{n\to \infty} \P(\la_{z-1}(M)=\la) = \prod_{r=1}^\infty (1 -\frac{1}{q^r}) \frac{1}{ q^{\sum_i (\la_i')^2} \prod_i (1/q)_{m_i(\la)}},
	\end{equation}
	where $\la_{z-1}(M)$ is the partition corresponding to $z-1$ in the rational canonical form of a uniform matrix $M$. We refer the reader to Section \ref{section:uniform} for a precise definition of $\la_{\phi}(M)$.

	The measure above had been studied extensively in Number Theory in the context of the Cohen-Lenstra heuristics.  Indeed, assume that $B$ is a $p$-group, $B= \Z/p^{\la_1}\Z \times \dots \times \Z/p^{\la_m}\Z$, Friedman and Washington \cite{FW} showed that for a Haar random matrix $M$ in $\Z_p$
	$$\lim_{n \to \infty } \P(M (\Z_p^n) / \Z_p^n \simeq B) = M_{p}(\la) = \frac{\prod_{k=1}^\infty (1-p^{-k})}{|\Aut(B)|}.$$
	We also refer the reader to \cite{W1,W2} and the references therein for related results.
	
	In the spirit of \eqref{eqn:z-1}, Fulman \cite{Fulman-thesis, Fulman-BAMS} also showed for the uniform model that as $n \to \infty$, for any fixed irreducible polynomial $\phi$, 
	$$\lim_{n\to \infty} \P(\la_{\phi}(M)=\la) = M_{q^{\deg(\phi)}}(\la)$$
	and in particularly, 
	\begin{equation}\label{eqn:poly}
	\P(\phi(x) | \det (M-x)) = \P(\la_{\phi}(M) \neq \emptyset) \to 1- \prod_{i=1}^\infty (1-q^{ i \deg(\phi)}).
	\end{equation}
	Moreover, it was also shown that these statistics are asymptotically independent for different $\phi$ in the sense that for any different irreducible polynomials $\phi_1,\dots, \phi_k$
	\begin{equation}\label{eqn:independence}
	\lim_{n\to \infty} \P(\wedge_{i=1}^k \la_{\phi_i} = \la_i) = \prod_{i=1}^k \lim_{n \to \infty}\P(\la_{\phi_i}= \la_i).
	\end{equation}
	
	We invite the reader to Section \ref{section:uniform} for a useful tool to deduce Equations \eqref{eqn:free},  \eqref{eqn:z-1} and \eqref{eqn:poly}.

	\subsection{Our main results} Motivated by the universality phenomenon in Random Matrix Theory, we wonder if the above statistics also hold for other models of $M$. While there have been many results addressing universality of random matrices in characteristic zero (to study the spectral behavior of various models of random matrices), we have not seen much in the literature addressing universality behavior in the finite fields setting. In fact, to the best of our knowledge, although there had been partial results such as \cite{Balakin, BKW, BM, Cooper, KK, CRR, KL1, KL2, SW}, universality results of matrices in finite fields only appeared very recently in  \cite{M1,M2, NgP, NgW1, W1,W2}. For instance, regarding the rank distribution, a simple consequence of results of Maples \cite{M1,M2} (see also \cite{NgP}) and of Wood \cite{W1} showed  
	\begin{theorem}\label{thm:z-1:universality} Let $\al>0$ be fixed. Assume that $M=(m_{ij})$ is a random matrix where $m_{ij}$ are iid copies of a random $\al$-balanced distribution in $\F_p$. Then we have
$$\lim_{n\to \infty} \P(M \mbox{ has rank } n -k) = \frac{1}{p^{k^2}} \frac{\prod_{i=k+1}^\infty (1-p^{-i})}{ \prod_{i=1}^{k} (1-p^{-i}) }.$$
	\end{theorem}
	Here we say that a random variable $\xi$ in $\F_p$ is $\al$-balanced \footnote{For simplicity, our notion here is weaker than those from \cite{M1,M2, NgP, NgW1} in that $\al$ is fixed.} if 
	$$\max_{a}\P(\xi=a)\le 1-\al.$$
	The method of \cite{M1,M2} (see also \cite{NgP}) relies on a swapping technique from \cite{KKSz,TV,BVW}, and can yield exponentially small error bound of type $\exp(- c \al n)$ in Theorem \ref{thm:z-1:universality}. However this technique is quite delicate, and does not seem to extend to other interesting models of matrices such as symmetric matrices. The method of \cite{W1,W2} mainly rely on the moment method, which extends rather easily to matrices of entries over $\Z/r\Z$ for composite $r$ (and to control other algebraic statistics beside the ranks), but one has to assume $r$ to be sufficiently small.

	One of the main goals of this note is to provide three alternative methods, which we will call the ``arithmetic approach" (after \cite{NgV,TVcir}),``geometric approach" (after \cite{RV}) and ``combinatorial approach" (after \cite{FJLS}). Although the error bounds obtained by these methods are usually of subexponential type (rather than exponential type), we believe that the methods will be extremely useful for the study of random matrices in finite fields. For instance the methods can be adapted to matrices with constraints such as symmetric matrices and antisymmetric matrices \cite{NgW2} to answer a question from \cite{CLKPW}. To highlight a result of this method, we show in Section \ref{section:normal:non-structure} the following result

	\begin{theorem}[Rank distribution]\label{thm:Maple'}
	
	Assume that $0\le d  \le n^c$ for a sufficiently small constant $c$. Assume that $p \le \exp(n^{c})$. Then for a random $n \times n$ matrix $M$ with entries being iid copies of an $\a$-balanced random variable $\xi$ in $\F_p$ we have 
		$$\P(\rank(M) = n-d)) =  \frac{1}{p^{d^2}} \frac{\prod_{i=d+1}^\infty (1-p^{-i})}{\prod_{i=1}^{d} (1-p^{-i}) } + O(e^{-n^{c'}}),$$
		where $c'$ is another (sufficiently small) positive constant depending on $c$ and $\al$.
	\end{theorem}
	
	Note that we can also establish similar rank distribution for rectangular matrices of size $(n+u) \times n$ for a fixed $u$ by a similar method. These results are not new and weaker than existing results in the literature (see for instance \cite{M1,M2} and \cite[Theorem A.4]{NgP}, \cite[Theorem 5.3]{NgW1}.)  However, as mentioned, our approach is new and seems to be robust. (For instance it can be used to prove Theorems \ref{thm:structure:iid}, \ref{thm:entrywise}, \ref{thm:char:universal} and \ref{thm:free:universal} below.) More precisely, to establish Theorem \ref{thm:Maple'} we will analyze the normal vectors $\Bw=(w_1,\dots,w_n)$ of random subspaces for which we will show that the random sum $\sum_i \xi_i w_i$ spreads out in $\F_p$ uniformly very quickly.
	
	\begin{theorem}[Non-structure of the normal vectors]\label{thm:structure:iid} With the same assumption as in Theorem \ref{thm:Maple'}, let $X_1,\dots, X_{n-d}, X_{n-d+1}$ be the first $n-d+1$ column vectors of $M$. Let $\Bw=(w_1,\dots, w_n)$ be any non-zero vector that is orthogonal to $W_{n-d} = \langle X_1,\dots, X_{n-d}\rangle$. Then with probability at least $1 -\exp(-\Theta(n))$ with respect to $X_1,\dots, X_{n-d}$  we have
		$$\sup_{a\in \F_p}| \P(\sum \xi_i w_i =a ) -1/p| \le \exp(-n^{c}),$$
		where $\xi_i$ are iid copies of $\xi$.
	\end{theorem}
	
We will prove the above result by showing that with very high probability the normal vectors do not have any structure (in any arithmetic, geometric, or combinatorial sense). On the other hand, we will also show that the normal vectors actually behave like a uniform vector in $\F_p^n$. This can be seen as a discrete analog of \cite{NgV-normal} where it was shown that normalized normal vectors of a random hyperplane in $\R^n$ behave like a uniform vector on the unit sphere.

	\begin{theorem}[Uniformity of the normal vectors]\label{thm:entrywise} With the same assumption as in Theorem \ref{thm:Maple'}, and conditioning on the event that the subspace $W_{n-1}$ generated by $X_1,\dots, X_{n-1}$ has full rank, we have
		\begin{itemize}
			\item For each $i \in \{1, \cdots, n\}$,
			$$|\P(w_i = 0) - 1/p| \leq O(\exp(-n^{c'})).$$
			\item For each $i \neq j$, and for any $a\in \F_p$ we have
			$$|\P(\exists \Bw=(w_1,\dots, w_n) \in W_{n-1}^{\perp}: w_i = a \wedge w_j=1) - 1/p| \leq O(\exp(-n^{c'})).$$
			
			\noindent 
		
			\vskip .05in
			\item Furthermore, with $n_a$ being the number of $i$ such that $w_i=a$, if we assume $\delta<1$ and $\delta^{-2}p=o(n/\log n)$ then
			\begin{equation}\label{eqn:na}
			\P(\wedge_{a=0}^{p-1}(|n_a/n - 1/p| \le \delta/p) \ge 1 -e^{-c \delta^2 n/ p}.
			\end{equation} 
			where $c>0$ is absolute. 
		\end{itemize}
	\end{theorem}
	We need $p$ to be smaller than $n$ in Eq. \eqref{eqn:na} so that $n_a$ is not vanishing on average. We remark that the above result holds trivially for the uniform model (when $X_1,\dots, X_{n-1}$ a chosen uniformly at random from $\F_p^n$) as in this case $\Bw$ is distributed as a uniform vector. However, it is not clear at all as to why $\Bw$ also behaves like a random uniform vector even when the $X_i$ are sampled differently.

	In the above results there is a natural connection between the ranks and the normal vectors. Somewhat more surprisingly, we show that these quantities can also be used to study the characteristics polynomials. Namely we can obtain the following analog of Equation \eqref{eqn:poly}.
	
	\begin{theorem}[Divisibility of the characteristic polynomials]\label{thm:char:universal} With the same assumption as in Theorem \ref{thm:Maple'}, let $\eps<1$ be any fixed constant. For a prime $p$ and fixed polynomial $\phi$ such that $C_{\phi} \le p \le n^{(1-\eps)/2}$ and $\phi$ is irreducible over $\F_p$ we have
		$$\Big |\P(\phi(x)|\det(M-x)) - [1- \prod_{i=1}^\infty (1-p^{- i \deg(\phi)})] \Big|  =O (\exp(-c n^{1-\eps}/p^2)),$$
		where $c$ and the implied constant depend on $\eps$ and $C_{\phi}$ is a constant that depends only on the degree of $\phi$.
	\end{theorem}
	
	Also, we will show the following analog of Equation \eqref{eqn:free}.
	
	\begin{theorem}[Universality for eigenvalue-free matrices]\label{thm:free:universal} Assume that $M$ is as in Theorem \ref{thm:Maple'}. We have
		$$\lim_{p\to \infty} \lim_{n\to \infty} \P(M \mbox{ is eigenvalue-free}) = 1/e.$$
	\end{theorem}
	
	Thus, for instance, our result works for the following simple-looking model of random matrices of (mean zero) integral entries. Let $A\subset \Z$ be a finite deterministic set of integers (such as $A=\{-1,0,3\}$ or $A=\{-7,-1,0,1,2\}$, etc.), and let $p$ be a prime so that the projection $\pi(A)$ of $A$ onto $\Z/p\Z$ is not a single point. Let $\xi$ be the image of the uniform measure on $A$ under $\pi$, then with the same notations as above we have the following result.
	
	\begin{corollary} Among all $|A|^{n^2}$ matrices whose entries are all in $A$, there are
		\begin{itemize}
			\item  $(\frac{1}{p^{d^2}} \frac{\prod_{i=d+1}^\infty (1-p^{-i})}{\prod_{i=1}^{d} (1-p^{-i}) } + O(e^{-n^{c'}}))$-portion of them have rank $n-d$ in $\F_p^n$;
			\vskip .1in
			\item $(1- \prod_{i=1}^\infty (1-p^{- i \deg(\phi)}) +O(\exp(-c n^{1-o(1)}/p^2)))$-portion of them have characteristic polynomial divisible by a given irreducible polynomial $\phi(x)$ for primes $p\lesssim n^{1-o(1)}$;
			\vskip .1in
			\item $(e^{-1}+o(1))$-portion of them are eigenvalue-free as $n\to \infty $ and $p \to \infty$;
			\vskip .1in
			\item the normal vectors satisfy Theorem \ref{thm:structure:iid} and \ref{thm:entrywise}.
		\end{itemize}
	\end{corollary}

	\subsection{Notation} We write $\P$ for probability and $\E$ for expected value. For an event $\mathcal{E}$, we write $\bar{\mathcal{E}}$ for its complement. We write $\exp(x)$ for the exponential function $e^x$.
	We use $[n]$ to denote $\{1,\dots,n\}$.  
	
	For a given index set $J \subset [n]$ and a vector $\Bx= (x_1,\dots, x_n)$, we write $\Bx|_J$ to be the subvector of $\Bx$ of components indexed from $J$. Similarly, if $H$ is a subspace then $H|_J$ is the subspace spanned by $\Bx|_J$ for $ \Bx\in H$.  
	
	For a vector $\Bw=(w_1,\dots,w_n)$ we let $\supp(\Bw)=\{i\in[n] | w_i\ne 0\}$. We will also write $\Bx \cdot \Bw$ for the dot product $\sum_{i=1}^n x_i w_i$.
	We say $\Bw$ is a \emph{normal} vector for a subspace $H$ if $\Bx\cdot \Bw=0$ for every $\Bx\in H$.
	
	For $I, J \subset [n]$, the matrix $M_{I \times J}$ is the submatrix of the rows and columns indexed from $I$ and $J$ respectively. Sometimes we will also write $M_n$ for $M_{n\times n}$ if there is no confusion.
	
	We write $\|.\|_{\R/\Z}$ to be the distance to the nearest integer. Sometimes, for a matrix $M$ we write $\row_i(M)$ and $\col_i(M)$ for this $i$-th row and column respectively. We write $X =
	O(Y)$, $Y=\Omega(X)$, $X \lesssim Y$, or $Y \gtrsim X$ if $|X| \leq CY$ for some fixed $C$. We also write that $X \asymp Y$ if $X \lesssim Y$ and $Y \gtrsim X$.

	Our paper is organized as follows. We will first discuss tools to prove Equations  \eqref{eqn:free}, \eqref{eqn:z-1} and \eqref{eqn:poly} in Section \ref{section:uniform}. We will present our characterization methods in Sections \ref{section:GAP}, \ref{section:ULCD}, \ref{section:combinatorial}, and then use these results to prove Theorem \ref{thm:structure:iid} in Section \ref{section:normal:non-structure}. We will present a short proof of Theorem \ref{thm:Maple'} in Section \ref{section:rank} and of Theorem \ref{thm:entrywise} in Section \ref{section:normal:equi}. The remaining two sections are reserved to prove Theorem \ref{thm:char:universal} and Theorem \ref{thm:free:universal} respectively.

	\section{The uniform model}\label{section:uniform}

	In this part we discuss the method to prove Equations \eqref{eqn:free}, \eqref{eqn:z-1} and \eqref{eqn:poly}. Although this is not the main goal of the note, we would like to present it here for pedagogical purposes, as for most of the cases the universal statistics are computed from the uniform model. We refer the reader to \cite{Fulman-BAMS} for a comprehensive survey on the method and its other applications.
	
	We first introduce a simple representative for $\Mat(n, q)$ (or $GL(n, q)$) modulo the conjugacy action of $GL(n,q)$. To motivate the formulas, let us first introduce a simpler variant for the permutation groups $S_n$. For a permutation $\pi$ let $n_i(\pi)$ be the number of cycles of length $i$ of $\pi$. The cycle index of a subgroup $G$ of $S_n$ is defined as $\frac{1}{|G|} \sum_{\pi \in G} \prod_{i\ge 1}  x_i^{n_i(\pi)}$. The function $f(u,\Bx) = 1+ \sum_{n\ge 1} \frac{u^n}{n!}  \sum_{\pi \in S_n} \prod_{i\ge 1}  x_i^{n_i(\pi)}$  is called the {\it cycle index generating function} of the symmetric groups and P\'olya's result shows that $f(u, \Bx) = \prod_{m\ge 1} e^{\frac{x_m u^m}{m}}$. This formula is useful in the study of (conjugacy) class functions of permutations. 
	
	For matrices over $\F_q$, the cycle index generating functions can be described by first giving some information on the conjugacy classes. Let $\la$ be a partition of some non-negative integer $|\la|$ into integer parts $\la_1\ge \la_2 \ge \dots \ge 0$. In what follows $m_i(\la)$ denotes the number of parts of $\la$ of size $i$, $\la'$ is the partition dual to $\la$, and $(u/q)_i$ denotes $(1-u/q) \cdots (1-u/q^i)$.
	
	Recall that we define the characteristic polynomial of an $n \times n$ matrix $M$ as $\det(M-x)$. 
	Assume that the irreducible decomposition of the characteristic polynomial of a matrix $M$ has the form $\det(M -x) = \prod_{i=1}^k \phi_i(x)^{\lambda_{\phi_i}}$. The {\it rational canonical form} of the conjugacy class containing $M$ is a matrix of form  
	$$\begin{pmatrix} R_1 & 0 & 0 & \dots & 0 \\ 0 & R_2 & 0 & \dots & 0  \\ \dots & \dots & \dots & \dots & \dots  \\ 0 & 0 & \dots & \dots & R_k  \end{pmatrix} .$$
	where each matrix $R_i$ has the form
	$$R= \begin{pmatrix} C(\phi_i^{\la_1}) & 0 & 0 & \dots & 0 \\ 0 &  C(\phi_i^{\la_2}) & 0 & \dots & 0  \\ \dots & \dots & \dots & \dots & \dots  \\ 0 & 0 & \dots & \dots &  C(\phi_i^{\la_m})  \end{pmatrix},$$
	and $\sum_i \la_i =\lambda_{\phi_i}$. Also we have the constraint that $\sum_{i=1}^k \lambda_{\phi_i} deg(\phi_i) = n$. Here for $\phi(x)=\sum_{i=0}^{d-1} a_i x^i + x^d$, the companion matrix $C(\phi)$ is defined as
	$$C(\phi):=\begin{pmatrix} 0 & 1 & 0 & 0 & \dots & 0 \\ 0 & 0 & 1 & 0 & \dots & 0  \\  0 & 0 & 0 & 1 & \dots & 0  \\ \dots & \dots & \dots & \dots & \dots  & \dots  \\ 0 & 0 & \dots & 0 & 0&  1 \\ -a_0 & -a_1 & -a_2 & \dots & -a_{d-2} & -a_{d-1}  \end{pmatrix} .$$
	
	In other words, we have the decomposition of  $\F_q^n= \oplus_\phi V_\phi$ where the characteristic polynomial of $\al$ on $V_\phi$ is $\phi^k$,  and furthermore $V_\phi = \oplus_i V_i$ where $V_i$ are {\it cyclic subspaces} with dimension $\la_i \deg(\phi)$.
	
	Note that in the data given above, each irreducible polynomial $\phi$ is assigned a partition $\la_\phi(\al)$. For example for $M=I_n$, then $\la_{z-1}=(1,1,\dots, 1)$, and $\la_\phi = \emptyset$ for all other $\phi$; while for 
	$$M = \begin{pmatrix} 1 & 1 & 0 & \dots & 0 \\ 0 & 1 & 0 & \dots & 0  \\ \dots & \dots & \dots & \dots & \dots  \\ 0 & 0 & \dots & \dots & 1  \end{pmatrix}$$ then $\la_{z-1}=(2,1,\dots, 1)$, and $\la_\phi = \emptyset$ for all other $\phi$.

	To introduce the {\it cycle index formula} for $Mat(n,q)$, let $x_{\phi, \la}$ be variables corresponding to pairs of polynomials and partitions. Define 
	$$Z_{Mat(n,q)} := \frac{1}{|GL(n,q)|} \sum_{\a \in Mat(n,q)} \prod_{\phi, |\la_\phi(\al)| >0} x_{\phi, \la_\phi(\al)}.$$

	Beautiful results of Kung \cite{Kun} and Stong \cite{Stong} showed that 
	$$1 + \sum_{n=1}^\infty Z_{Mat(n,q)} u^n = \prod_{\phi} [ 1 + \sum_{n\ge 1} \sum_{\la \vdash n} x_{\phi, \la} \frac{u^{n \deg(\phi)}}{ q^{\deg(\phi) \sum_i (\la_i')^2 \prod_{i\ge 1} (\frac{1}{q^{\deg(\phi)}})_{m_i(\la_\phi)}} }].$$
	Note that one can also define $Z_{GL(n,q)}$ similarly. The above formula allows one to study class functions for matrices over $\F_q$, for which we now give a proof for Equation \eqref{eqn:free}, a proof for Equations \eqref{eqn:z-1} and \eqref{eqn:poly} can be done similarly by specifying the variables $x_{\phi, \lambda}$ appropriately.
	
	\begin{proof}(of Equation \eqref{eqn:free}) In the cycle index formula above, by specializing the variables $x_{\phi,\lambda}$ we may count different subsets of $Z_{Mat(n,q)}$. For instance if we set $x_{\phi,\lambda} = 1$ we get everything so,
		\begin{equation}\label{eq:countall}
		\prod_{\phi} [ 1 + \sum_{n\ge 1} \sum_{\la \vdash n}  \frac{u^{n \deg(\phi)}}{ q^{\deg(\phi) \sum_i (\la_i')^2 \prod_{i\ge 1} (\frac{1}{q^{\deg(\phi)}})_{m_i(\la_\phi)}} }] = (1-u)^{-1}.\end{equation}
		
		We want to count matrices with no fixed subspace. In terms of $x_{\phi,\la}$ this is the same as $x_{\phi,\lambda}=0$ for linear $\phi$ and $x_{\phi,\lambda}=1$ otherwise. Making this assignment and using \eqref{eq:countall} we have,
		\begin{equation*}
		1 + \sum_{n=1}^\infty \frac{H_{n,q}}{|GL(n,q)|}u^n =\Big(1 +  \sum_{i=1}^\infty\sum_{\lambda\vdash i} \frac{u^i}{  q^{\sum_i (\la_i')^2 \prod_{i\ge 1} (\frac{1}{q})_{m_i(\la_\phi)}}  }\Big)^{1-q}(1-u)^{-1}\label{der},
		\end{equation*}
		where $H_{n,q}$ is the number of derangements in $GL(n,q)$. Now the $n^{th}$ coefficient of this generating function is going to the first term in the product evaluated at $u=1$ and by a result of Fine-Herstein \cite{FinHer}  we have (with $[q]_i=\prod_{k=0^{i-1}}(q^i-q^k)$),
		\begin{equation*}
		\frac{H_{n,q}}{|GL(n,q)|} \to \Big(1 + \sum_{i=1}^\infty \frac{q^{i(i-1)}}{[q]_i}\Big)^{1-q} \text{ as } n\to\infty.
		\end{equation*}
		
		Some cursory analysis (using the fact that the asymptotic behavior of the sum is determined by its first term) shows
		\begin{equation*}
		\Big(1 + \sum_{i=1}^\infty \frac{q^{i(i-1)}}{[q]_i}\Big)^{1-q} \to 1/e \text{ as } q\to\infty,
		\end{equation*}
		as desired. \end{proof}

	\section{Structures of vectors in $\F_p^n$: an almost optimal characterization}\label{section:GAP} 
	Let $G$ be an (additive) abelian group. A set $Q$ is a \emph{generalized arithmetic progression} (GAP) of
	rank $r$ if it can be expressed in the form
	$$Q= \{a_0+ x_1a_1 + \dots +x_r a_r| M_i \le x_i \le M_i' \hbox{ and $x_i\in\Z$ for all } 1 \leq i \leq r\}$$
	for some elements $a_0,\ldots,a_r$ of $G$, and for some integers $M_1,\ldots,M_r$ and $M'_1,\ldots,M'_r$.
	One can think of Q as the image of an integer box $B = \{(x_1, . . . , x_r) \in \Z^r|M_i \le x_i \le M_i' \}$ under the linear map
	$$
	\Phi: (x_1, . . . , x_r) \rightarrow a_0 + x_1a_1 + \cdots + x_ra_r.
	$$
	Given $Q$ with a representation as above, the numbers $a_i$ are  \emph{generators} of $Q$, the numbers $M_i$ and $M_i'$ are  \emph{dimensions} of $Q$, and $\Vol(Q) := |B|$ is the \emph{volume} of $Q$ associated to this presentation (i.e. this choice of $a_i,M_i,M_i'$). We say that $Q$ is \emph{proper} for this presentation if the above linear map is one to one, or equivalently if $|Q| =|B|$. For an integer $t \geq 1$, we let $Q_t$ denote the dilation of $Q$ by $t$, i.e.
	$$
	Q_t = \{a_0+ x_1a_1 + \dots +x_r a_r| tM_i \le x_i \le tM_i' \hbox{ and $x_i\in\Z$ for all } 1 \leq i \leq r\}
	$$
	and we say $Q$ is \emph{$t$-proper} if $Q_t$ is also proper.  If $-M_i=M_i'$ for all $i\ge 1$ and $a_0=0$, we say that $Q$ is {\it symmetric} for this presentation.  A {\it coset progression} in $G$ is a set of type $H+Q$, where $H$ is a subgroup of $G$.
	
	Our main result here is that, if a random walk in $\Z/p\Z$ does not spread out evenly fast, then the steps must be arithmetically correlated (and vice versa).
	
	\begin{theorem}[Arithmetic structure, characterization I]\label{theorem:ILO}
		Let $\eps < 1$ and $C$ be positive constants. Suppose $\mu$ is a random variable that is $\alpha$-balanced taking values in $\Z/p\Z$ and that $\Bw = (w_1, \cdots, w_n) \in (\Z/p\Z)^n$ is such that 
		$$\rho(\Bw) = \sup_{a \in \Z/p\Z} |\P(\mu_1w_1 + \cdots + \mu_nw_n = a)-\frac{1}{p}| = \Theta(n^{-C}),$$ 
		where $\mu_1, \cdots, \mu_n$ are independent and identically distributed copies of $\mu$ and $p$ is an odd prime possibly depending on $n$. Then for any $n^{\epsilon/2} \leq n' \leq n$, there is a set $W'$ of $n-n'$ components $w_i$ such that one of the following holds.
		\begin{itemize}
			\item For $p \lesssim n^{C}$, there exists a GAP of rank one $Q$ that contains $W'$, where 
			$$|Q| \leq  O_{C, \eps}(1 + p \sqrt{(\log n) /n'}).$$ 
			\item For $p\gtrsim n^{C}$, there exists a proper symmetric GAP $Q$ of rank $r = O_{C,\epsilon}(1)$ that contains $W'$, where $$|Q| =O_{C, \epsilon}(1+\rho^{-1}/{n'}^{r/2}).$$

		\end{itemize}
	\end{theorem}
	Note that our characterization is almost optimal in the sense that it nearly implies the backward direction: if $\Bw$ satisfies the conclusion of the theorem, then $\rho =\Omega(n^{-C})$. It also implies that for $\rho\lesssim n^{-1/2 +o(1)}$ (for any $p$), recovering a result by Maples (see Theorem \ref{theorem:LO}. This is because that if a positive portion of the $w_i$ are non-zero, then the set $Q$ must have size at least $1$ and $r\ge 1$.)
 Our presentation here follows from \cite{NgW1} with some modifications (as in \cite{NgW1} we focused only on large $p$, and on the quantity $\sup_{a \in \Z/p\Z} \P(\mu_1w_1 + \cdots + \mu_nw_n = a)$ rather than on $\rho$ as above.) We will make use of two results from \cite{TVjohn} by Tao and Vu. The first result allows one to pass from coset progressions to proper coset progressions without any substantial loss. 
	
	\begin{theorem}\cite[Corollary 1.18]{TVjohn}\label{thm:proper} There exists a positive integer $C_1$ such that the following statement holds. Let $Q$ be a symmetric coset progression of rank $d\ge 0$ and let $t\ge 1$ be an integer. Then there exists a $t$-proper symmetric coset progression $P$ of rank at most $d$ such that we have
		$$Q \subset P \subset Q_{{(C_1d)}^{3d/2}t}.$$ 
		We also have the size bound
		$$|Q| \le |P| \le t^d {(C_1d)}^{3d^2/2} |Q|.$$
	\end{theorem}
	
	The second result, which is directly relevant to us, says that as long as $|kX|$ grows slowly compared to $|X|$, then it can be contained in a structure. This is a long-range version of the Freiman-Ruzsa theorem.
	
	\begin{theorem}\label{thm:longrange}\cite[Theorem 1.21]{TVjohn} There exists a positive integer $C_2$ such hat the following statement holds: whenever $d,k\ge 1$ and $X \subset G$ is a non-empty finite set such that 
		$$k^d|X| \ge 2^{2^{C_2 d^2 2^{6d}}} |kX|,$$
		then there exists a proper symmetric coset progression $H+Q$ of rank $0\le d'\le d-1$ and size $|H+Q| \ge 2^{-2^{C_2 d^2 2^{6d}}} k^{d'}|X|$ and $x,x' \in G$ such that $$x + (H+Q) \subset kX \subset x' + 2^{2^{C_2 d^2 2^{6d}}} (H+Q).$$
		
	\end{theorem}

	Note that any GAP $ Q=\{a_0+ x_1a_1 + \dots +x_r a_r :  -N_i \le x_i \le N_i \hbox{ for all } 1 \leq i \leq r\}$ is contained in a symmetric GAP  $ Q'= \{x_0 a_0+ x_1a_1 + \dots +x_r a_r : -1\le x_0 \le 1,  -N_i \le x_i \le N_i \hbox{ for all } 1 \leq i \leq r\}$. Thus, by combining Theorem \ref{thm:longrange} with Theorem \ref{thm:proper} we obtain the following
	
	\begin{corollary}\label{cor:longrange} Whenever $d,k\ge 1$ and $X \subset G$ is a non-empty finite set such that 
		$$k^d|X| \ge 2^{2^{C_2 d^2 2^{6d}}} |kX|,$$
		then there exists a 2-proper symmetric coset progression $H+P$ of rank $0\le d'\le d$ and size $|H+P| \le 2^d (C_1 d)^{3d^2/2} 2^{d2^{C_2 d^2 2^{6d}}} |kX|$ such that 
		$$ kX \subset H+P.$$
		
	\end{corollary}
	
	
	\begin{proof}(of Theorem \ref{theorem:ILO})
		First, for convenience we will pass to symmetric distributions. Let $\psi = \mu-\mu'$ be the symmetrization and let $\psi'$ be a lazy version of $\psi$ so that 
		\[
		\P(\psi'=x) = \begin{cases}
		\frac{1}{2}\P(\psi=x) \mbox{ if } x\neq 0\\ 
		\frac{1}{2}\P(\psi=x) + \frac{1}{2}, \mbox{ if } x=0.
		\end{cases}
		\]
		Notice that $\psi'$ is symmetric as $\psi$ is symmetric. We can check that $\max_x \P(\psi =x)  \le 1 -\alpha$, and so
		$$ \sup_x \P(\psi'=x) \le 1 -\alpha/2.$$
		We assume that $\P(\psi' =t_j) = \P(\psi' =-t_j) = \beta_j/2>0$ for $1\le j\le l, t_j \neq 0$, and that $\P(\psi'=0) = \beta_0$, where $t_{j_1} \pm t_{j_2} \neq 0 \mod p$ for all $1 \le j_1, j_2 \le l$ and $j_1 \neq j_2$.
		Denote $S=\mu_1 w_1 +\dots + \mu_n w_n$. Consider $a  \in \Z/p\Z$ where  $\P( S=a)$ is maximum (or minimum). Using the standard notation $e_p(x)$ for $\exp(2\pi \sqrt{-1} x/p )$, we have
		\begin{equation}\label{eqn:fourier0} \P(S=a)= \E \frac{1}{p} \sum_{x \in \Z/p\Z} e_p (x (S-a)) = \E \frac{1}{p} \sum_{x \in \Z/p\Z} e_p (\xi S) e_p(-x a) = \frac{1}{p}+ \E \frac{1}{p} \sum_{x \in \Z/p\Z, x \neq 0} e_p (x S) e_p(-x a). 
		\end{equation}
		So
		\begin{equation}\label{eqn:fourier1}
		\rho = |\P(S=a)- \frac{1}{p} |\le \frac{1}{p} \sum_{  x\in \Z/p\Z, x \neq 0} |\E e_p (x S)|.
		\end{equation}
		
		By independence
		\begin{align*} \label{eqn:fourier2}  
		|\E e_p(x S)| &= \prod_{i=1}^n |\E e_p(x \mu_i w_i)|\le \prod_{i=1}^n (\frac{1}{2}(|\E e_p(x \mu_i w_i)|^2+1)) = \prod_{i=1}^n |\E e_p( x \psi'  w_i)| \\
		& = \prod_{i=1}^n ( \beta_0 +\sum_{j=1}^l \beta_j  \cos \frac{2\pi x t_j w_i}{p}).  \end{align*}
		It follows that
		\begin{align}
		\rho  & \le \frac{1}{p} |\sum_{x \in \Z/p\Z, x \neq 0}  \prod_{i=1}^n (\beta_0 + \sum_{j=1}^l \beta_j  \cos \frac{2\pi x t_j w_i}{p}) | \\ \nonumber
		& \le \frac{1}{p} \sum_{x \in \Z/p\Z, x \neq 0}  \prod_{i=1}^n (\beta_0 + \sum_{j=1}^l \beta_j  |\cos \frac{\pi x t_j w_i}{p} |)  , 
		\end{align}
		where we made the change of variable $x \rightarrow x /2$ (in $\Z/p\Z$) and used the triangle inequality. 
		
		By convexity, we have that   $|\sin  \pi z | \ge 2 \|z\|$ for any $z\in \R$, where $\|z\|:=\|z\|_{\R/\Z}$ is the distance of $z$ to the nearest integer. Thus,
		\begin{equation} \label{eqn:fourier3-1}| \cos \frac{\pi x}{p}|  \le  1- \frac{1}{2} \sin^2 \frac{\pi x}{p}  \le 1 -2 \|\frac{x}{p} \|^2. \end{equation}
		Hence for each $w_i$
		$$\beta_0 + \sum_{j=1}^l \beta_j  |\cos \frac{\pi x t_j w_i}{p} | \le 1 - 2 \sum_{j=1}^l \beta_j \|\frac{x t_j w_i}{p} \|^2 \le \exp(-2 \sum_{j=1}^l \beta_j \|\frac{x t_j w_i}{p} \|^2).$$
		Consequently, we obtain a key inequality
		\begin{equation} \label{eqn:fourier4}
		\rho \le  \frac{1}{p} \sum_{x \in \Z/p\Z, x \neq 0}  \prod_{i=1}^n (\beta_0 + \sum_{j=1}^l \beta_j  |\cos \frac{\pi x t_j w_i}{p} |)   \le  \frac{1}{p} \sum_{x \in \F_p, x \neq 0} \exp( - 2 \sum_{i=1}^n \sum_{j=1}^l \beta_j \|\frac{x t_j w_i}{p} \|^2).
		\end{equation}
		
		{\it Large level sets.}  Now we consider the level sets $S_m:=\{x | x \neq 0 \wedge \sum_{i=1}^n \sum_{j=1}^l \beta_j \|\frac{x t_j w_i}{p} \|^2  \le m  \} $.  We have
		$$n^{-C} \le \rho  \le   \frac{1}{p} \sum_{x \in F_p, x \neq 0} \exp( - 2 \sum_{i=1}^n \sum_{j=1}^l \beta_j \|\frac{x t_j w_i}{p} \|^2)  \le \frac{1}{p} \sum_{m \ge 1} \exp(-2(m-1)) |S_m| .$$
		As $\sum_{m\ge 1} \exp(-m) < 1$, there must be a large level set $S_m$ such that
		\begin{equation} \label{eqn:level1} |S_m| \exp(-m+2) \ge  \rho  p. \end{equation}
		In fact, since $\rho \ge n^{-C}$, 
		we can assume that $m=O(\log n)$. The bound  $|S_m| \ge \exp(m-2) \rho  p$ guarantees that $S_m$ is non-empty.
		\vskip .1in
		Now we consider two cases.
		
		{\bf Case 1.} We assume $p \lesssim n^{C}$. We know that $S_m$ is non-empty, and hence there exists $x_0 \neq 0$ so that 
		$$\sum_{i=1}^n \sum_{j=1}^l \beta_j \|\frac{x_0 t_j w_i}{p} \|^2  \le m.$$
		Set
		\begin{equation}\label{ean:al'}
		\alpha' := \sum_{j=1}^l \beta_j =1-\beta_0.
		\end{equation}
		Then by definition of $\xi$, we have 
		$$\alpha' \ge \al/2 .$$
		Thus we can rewrite the above as 
		$$\sum_{j=1}^l \beta_j \sum_{i=1}^n  \|\frac{x_0 t_j w_i}{p} \|^2  \le 2 \al^{-1} m \sum_j \beta_j .$$
		Thus there exists an index $j_0$ so that $ \beta_{j_0} \sum_{i=1}^n  \|\frac{x_0 t_{j_0} w_i}{p} \|^2  \le 2 \al^{-1} m \beta_{j_0}$, that is 
		$$ \sum_{i=1}^n  \|\frac{x_0 t_{j_0} w_i}{p} \|^2  \le 2 \al^{-1} m. $$
		So, for most $w_i$
		\begin{equation} \label{eqn:double1:p}  \|\frac{x_0 t_{j_0} w_i}{p} \|^2  \le \frac{ 2 \al^{-1} m }{n'}.  \end{equation} 
		
		More precisely, by averaging, the set of $w_i$ satisfying \eqref{eqn:double1:p}
		has size at least $n-n'$.  We call this set $W'$. The set $\{w_1,\dots, w_n\}\backslash W'$ has size at most $n'$ and this is the
		exceptional set that appears in Theorem \ref{theorem:ILO}. By definition, for $w_i$ from this set we have 
		\begin{equation} \label{eqn:double1':p}  
		\|\frac{x_0 t_{j_0} w_i}{p} \|^2 = O(\frac{m }{n'}).  
		\end{equation} 
		
		Hence we have seen that, after a dilation by $x_0 t_{j_0}$, $W'$ belongs to the arithmetic progression $P$ of rank one and of size $O(p \sqrt{(\log n)/ n'})$,
		$$P=\{x \in \F_p, \|\frac{x}{p}\| = O(\sqrt{\frac{m}{n'}})\}.$$
		Notice that in this case we don't have to assume $n'\ge n^{\eps/2}$.
		

		{\bf Case 2.} We assume $p\gtrsim n^C$. By  double-counting we have
		$$ \sum_{i=1}^n \sum_{x \in S_m} \sum_{j=1}^l \beta_j \|\frac{x t_j w_i}{p} \|^2 =   \sum_{x \in S_m} \sum_{i=1}^n  \sum_{j=1}^l \beta_j \|\frac{x t_j w_i}{p} \|^2  \le m |S_m |.$$
		So, for most $w_i$,
		\begin{equation} \label{eqn:double1} \sum_{x \in S_m}\sum_{j=1}^l \beta_j \|\frac{x t_j w_i}{p} \|^2  \le \frac{m }{n'} |S_m| \end{equation} 
		for some large constant $C_0$.
		
		By averaging, the set of $w_i$ satisfying \eqref{eqn:double1}
		has size at least $n-n'$.  We call this set $W'$. The set $\{w_1,\dots, w_n\}\backslash W'$ has size at most $n'$ and this is the
		exceptional set that appears in Theorem \ref{theorem:ILO}. In the rest of the proof, we are going to show that $W'$ is a dense subset of a proper GAP.
		
		Since $\|\cdot \|$ is a norm, by the triangle inequality, we have  for any $a  \in k W' $
		\begin{equation} \label{eqn:double2} \sum_{x \in S_m} \sum_{j=1}^l \beta_j \|\frac{x t_j a}{p} \|^2  \le  k^2 \frac{m}{n'} |S_m| . \end{equation}
		More generally, for any $k'  \le k $ and $a \in k'W',$
		\begin{equation} \label{eqn:double3} \sum_{x \in S_m} \sum_{j=1}^l \beta_j \|\frac{x t_j a}{p} \|^2   \le  {k'}^2 \frac{m}{n'} |S_m| . \end{equation}
		
		{\it Dual sets.} 		Define  
		$$S_m^{\ast} :=\{ a | \sum_{x\in S_m}  \sum_{j=1}^l \beta_j \|\frac{x t_j a}{p} \|^2  \le \frac{ \al' }{200} |S_m |\}$$ 
		where the constant $200$ is ad hoc and any sufficiently large constant would do.
		We have
		\begin{equation} \label{eqn:dual1} |S_m^{\ast} |  \le \frac{8 p}{|S_m |} . \end{equation}
		To see this, define $T_a :=\sum_{x \in S_m} \sum_{j=1}^l \beta_j \cos \frac{2\pi a t_j x}{p}$. Using the fact that $\cos 2\pi z \ge 1 -100 \|z\|^2 $ for any $z \in \R$, we have, for any $a \in S_m^{\ast},$
		$$T_a \ge  \sum_{x \in S_m} (1- 100 \sum_{j=1}^l \beta_j \|\frac{x t_j a}{p} \|^2) \ge \frac{\al' }{2} |S_m |. $$
		One the other hand, using the basic identity $\sum_{a \in \Z/p\Z} \cos \frac{2\pi  a z}{p} = p\I_{z=0} $, we have (taking into account that $t_{j_1} \neq t_{j_2} \mod p$)
		$$\sum_{a \in \Z/p\Z} T_a^2 \le 2p |S_m| \sum_j \beta_j^2 \le  2p |S_m|  \max_{1\le j\le l} \beta_j (\sum_{j=1}^l \beta_j)  \le  2p |S_m|  {\al'}^2.$$
		Equation \eqref{eqn:dual1} then follows from the last two estimates and averaging.
		
		Next, for a properly chosen  constant $c_1$ we set
		$$k := c_1 \sqrt{\frac{\al' n'}{m}}.$$ 
		By \eqref{eqn:double3} we have $\cup_{k'=1}^k  k' W'  \subset  S_m^{\ast} $. Next, set 
		$$W^{''} := W' \cup \{0\}.$$ 
		We have $k W^{''} \subset S_m^{\ast} \cup \{0\} $. This results in the critical bound
		\begin{equation}  \label{eqn:dual2} |k W^{''} |  = O( \frac{p}{|S_m|}) = O(\rho^{-1} \exp(-m+2)) .  \end{equation}
		We are now in a position to apply Corollary \ref{cor:longrange} with 
		$X$ as the set of distinct elements of $W^{''}$. As $k = \Omega (\sqrt {\frac{\al_n' n'}{m}}) =\Omega(\sqrt{\frac{ \al_n' n'}{\log n}})$, 
		\begin{equation}\label{eqn:LO:grow}
		\rho^{-1} \le n^{C} \le k^{4C/\eps+1}.
		\end{equation}
		
		It follows from Corollary \ref{cor:longrange}  that $kX$ is a subset of
		a 2-proper symmetric coset progression $H+P$ of rank $r=O_{C, \epsilon_0} (1)$ and cardinality
		$$|H+P| \le O_{C,\eps} |kX|.$$
		Now we use the special property of $\Z/p\Z$ that it has only trivial proper subgroup. As $|kX| =O(n^C)$, and as $p\gtrsim n^C$, the only way that $|kX| \gtrsim |H+P|$ is that $H=\{0\}$. 
		Consequently,  $kX$ is now a subset of $P$, a 2-proper symmetric GAP of rank $r=O_{C, \epsilon_0} (1)$ and cardinality 
		\begin{equation}\label{eqn:P:size}
		|P| \le O_{C,\eps} |kX|.
		\end{equation}
		To this end, we apply the following dividing trick  from \cite[Lemma A.2]{NgV}.
		\begin{lemma}\label{lemma:divide} 
			Assume that $0 \in X$ and that $P=\{ \sum_{i=1}^r
			x_ia_i: |x_i|\le N_i\}$ is a  2-proper  symmetric GAP that contains $kX$. Then $X\subset \{\sum_{i=1}^r x_ia_i: |x_i|\le
			2N_i/k\}$.
		\end{lemma}
		
		Combining \eqref{eqn:P:size} and Lemma \ref{lemma:divide} we thus obtain a GAP $Q$ that contains $X$ and
		\begin{align*}
		|Q| = O_{C, \epsilon_0}(k^{-r}|kX|)= O_{C, \epsilon_0} (k^{-r} |kW^{''}|)&= O_{C, \epsilon_0} \left( \rho^{-1} \exp(-m) (\sqrt {\frac{ \alpha_n' n'}{m}})^{-r}\right)\\
		&= O_{C, \epsilon_0} (  \rho^{-1} (\alpha_n' n')^{-r} ),
		\end{align*}
		concluding the proof.
	\end{proof}

	Before concluding the section, we record here an elementary but useful result beyond the polynomial regime.
	
	\begin{theorem}[degenerate case]\label{theorem:degenerate}
		Let $\eps < 1$ and $\delta$ be positive constants such that $\delta<\eps$. Let $p$ be an odd prime number. Suppose $\mu$ is a random variable that is $\alpha$-balanced taking values in $\Z/p\Z$. Also, assume that $\Bw = (w_1, \cdots, w_n) \in (\Z/p\Z)^n$ is such that 
		$$\rho(\Bw) = \sup_{a \in \Z/p\Z} |\P(\mu_1w_1 + \cdots + \mu_nw_n = a)-\frac{1}{p}| \ge \exp(-n^\delta),$$ 
		where $\mu_1, \cdots, \mu_n$ are independent and identically distributed copies of $\mu$. Then for any $n^{\epsilon/2} \leq n' \leq n$, there is a set $W'$ of $n-n'$ components $w_i$ and then a GAP of rank one $Q$ that contains $W'$, where 
			$$|Q| \leq  O_{C, \eps}(p\sqrt{n^\delta /n'}).$$ 
			
				\end{theorem}
	We note that be bound on $Q$ above is very close the the trivial bound $p$. The result is effective for not too large $p$.
	
	\begin{proof} We proceed as in the proof of Theorem \ref{theorem:ILO} until \eqref{eqn:fourier4} that
	\begin{equation*}
		\rho  \le  \frac{1}{p} \sum_{x \in \F_p, x \neq 0} \exp( - 2 \sum_{i=1}^n \sum_{j=1}^l \beta_j \|\frac{x t_j w_i}{p} \|^2).
		\end{equation*}
		
		We recall the level sets $S_m=\{x | x \neq 0 \wedge \sum_{i=1}^n \sum_{j=1}^l \beta_j \|\frac{x t_j w_i}{p} \|^2  \le m  \} $.  We have
		$$\exp(-n^\delta) \le \rho  \le   \frac{1}{p} \sum_{x \in F_p, x \neq 0} \exp( - 2 \sum_{i=1}^n \sum_{j=1}^l \beta_j \|\frac{x t_j w_i}{p} \|^2)  \le \frac{1}{p} \sum_{m \ge 1} \exp(-2(m-1)) |S_m| .$$
		As $\sum_{m\ge 1} \exp(-m) < 1$, there must be a large level set $S_m$ such that
		\begin{equation} \label{eqn:level1} |S_m| \exp(-m+2) \ge  \rho  p. \end{equation}
		In fact, since $\rho \ge \exp(-n^\delta)$, 
		we can assume that $m=O(n^\delta)$. The bound  $|S_m| \ge \exp(m-2) \rho  p$ guarantees that $S_m$ is non-empty.
		\vskip .1in
		Our next step is almost identical to the proof of the first part of Theorem \ref{theorem:ILO}. As $S_m$ is non-empty, there exists $x_0 \neq 0$ so that 
		$$\sum_{i=1}^n \sum_{j=1}^l \beta_j \|\frac{x_0 t_j w_i}{p} \|^2  \le m.$$
		With $\al'$ as in \eqref{ean:al'} we have $\alpha' \ge \al/2$, and we can rewrite the above as 
		$$\sum_{j=1}^l \beta_j \sum_{i=1}^n  \|\frac{x_0 t_j w_i}{p} \|^2  \le 2 \al^{-1} m \sum_j \beta_j .$$
		Thus there exists an index $j_0$ so that $ \beta_{j_0} \sum_{i=1}^n  \|\frac{x_0 t_{j_0} w_i}{p} \|^2  \le 2 \al^{-1} m \beta_{j_0}$, that is 
		$$ \sum_{i=1}^n  \|\frac{x_0 t_{j_0} w_i}{p} \|^2  \le 2 \al^{-1} m. $$
		So, with $W'$ be the set of $w_i$ such that $\|\frac{x_0 t_{j_0} w_i}{p} \|^2  \le \frac{ 2 \al^{-1} m }{n'}$ then $W'$ has at least $n-n'$ elements. 
 	          By definition, for $w_i\in W'$ we have $ 
		\|\frac{x_0 t_{j_0} w_i}{p} \|^2 = O(\frac{m }{n'}).  
		$ and this implies that after a dilation by $x_0 t_{j_0}$ the set $W'$ belongs to the arithmetic progression $P$ of rank one 		
		$$P=\{x \in \F_p, \|\frac{x}{p}\| = O(\sqrt{\frac{m}{n'}})\}.$$
		Notice that the size of $P$ is bounded by $O(p \sqrt{(m)/ n'})=O(p \sqrt{n^\delta/ n'})$ as desired.

	\end{proof}

	\section{Structures of vectors in $\F_p^n$: a geometric approach}\label{section:ULCD}

	From now on, for simplicity we will assume our random variables are iid Bernoulli (taking values  $\pm 1$ with probability 1/2) and $p\ge 3$, the general $\al$-balanced case can be treated almost identically (see {Remark \ref{remark:inverse:modp:alpha}.)
		
		Let $\Bw=(w_1,\dots, w_n)$ be a non-zero vector in $\F_p^n$, where $p$ is a prime. We first cite a result of Erd\H{o}s-Littlewood-Offord type from \cite{M1}
		\begin{theorem}\label{theorem:LO} Let $c_{nsp}>0$ be a constant, and assume that 
			\begin{equation}\label{eqn:sparse}
			|\supp(\Bw)| \ge c_{nsp}n.
			\end{equation} 
			Then
			$$\rho(\Bw)=\sup_{r} |\P(X \cdot \Bw =r (\mod p)) -\frac{1}{p}| =O(\frac{1}{\sqrt{n}})$$
			where the implied constant depends on $c_{nsp}$, and where $X=(x_1,\dots, x_n)$ and $x_i$ are iid Bernoulli.
		\end{theorem}
		
		In what follows, if not specified, we always assume our deterministic vector $\Bw$ to satisfy the non-sparsity property \eqref{eqn:sparse}. We remark that this non-sparsity property passes to all other dilations $t\Bw$ of $\Bw$ in $\F_p^n$ for non-zero $t$.
		
		As mentioned in the introduction, our treatment in this section is motivated by the work of Rudelson and Vershynin (in characteristic zero) \cite{RV} and we hope to develop a ``geometric" characterization of the steps of our random walk in $\F_p$ if the walk spreads out slowly. This task is not straightforward; as we will see, there are many simple concepts in characteristic zero that are hard to find natural (and equally useful) analogs in the finite field setting (for instance, the notion of compressible and incompressible vectors).
		
		In some situations, if $\Bw=(w_1,\dots, w_n)$ is a vector in $\F_p^n$, then by viewing $\F_p$ as the interval $\I_p= [-(p-1)/2, (p-1)/2]$ in $\Z$, we will consider the components $w_i$ as integers from this interval. We then write $\Bw'$ as the vector in $\R^n$
		$$\Bw' =(w_1',\dots, w_n')= \frac{1}{p} \Bw = (w_1/p,\dots, w_n/p).$$   
		\begin{definition} Let $0<\gamma<1$ and $\kappa$ be given. Let $\Bw'=(w_1',\dots,w_n') \in \R^n$ be a non-zero vector in $\frac{1}{p}\Z^n$ where $\|\Bw'\|_\infty \le 1/2$. We denote by $\ULCD_{\gamma, \kappa}(\Bw')$ to be the smallest (infimum) positive integer $L$ such that 
			$$\dist(L \Bw', \Z^n) \le \min \{ \gamma \|L \Bw'\|_2, \kappa\}$$
		where $\dist(L \Bw', \Z^n)$ denotes the smallest Euclidean distance from $L \Bw'$ to an element of $\Z^n$.
		\end{definition}
		
		Throughout this paper, $\gamma<1$ is an absolute constant (such as $\gamma=1/8$), and $\kappa \le n^c$, for some positive constant $c\le 1/2$ to be chosen.
		
		This definition is in characteristic zero. Here we used the notion of $\ULCD$ (compared to the notion of $\LCD$ from \cite{RV}) to emphasize that $(w_1',\dots,w_n')$ is not normalized (i.e. its $\ell_2$-norm might not be unit). Notice that 
		$$2\le \ULCD_{\gamma,\kappa} \le p.$$ 
		Furthermore, if $ \ULCD_{\gamma,\kappa} = p $ then by definition we would have for all $1\le L \le p-1$ that
		\begin{equation}\label{eqn:p}
		\dist(L \Bw', \Z^n) \ge \min \{ \gamma \|L \Bw'\|_2, \kappa\}.
		\end{equation}

		\begin{remark}\label{remark:LCDlowerbound} Note that if for some $T>1$ we have $|w_i'|\le 1/2T$ for all $i$, then 
			$$\ULCD_{\gamma,\kappa}(\Bw') \ge T.$$ 
			This is because otherwise, $\|L w_i'\|_{\R/\Z} = |Lw_i'|$ and hence $\sum_i \|L w_i'\|_{\R/\Z}^2 = \sum_i \|L w_i'\|_2^2$, which cannot be smaller than $\gamma^2 \|L\Bw'\|_2^2$ by definition as $\gamma <1$. 
		\end{remark}
		
		Our result below says that if $\ULCD_{\gamma,\kappa}(\Bw')$ is large then the concentration probability is small. In our notation $t\Bw$ is another vector in $\F_p^n$, which again can be viewed as a vector in $\I_p^n =[-(p-1)/2, (p-1)/2]^n$. We then define $(t\Bw)'$ as $\frac{1}{p} t\Bw$ accordingly in this projection to characteristic zero.

		\begin{theorem}[Geometric structure, characterization II]\label{theorem:inverse:modp} Let $p\ge 3$ be a prime, and let $C>0$ be an arbitrary constant. Let $\Bw=(w_1,\dots, w_n)$ be a non-zero vector in $\Z^n$, where $|w_i|\le p/2$, and let $\Bw'=\frac{1}{p} \Bw = (\frac{w_1}{p},\dots, \frac{w_n}{p})$. Then 
			\begin{enumerate}
				\item If there is no non-zero $t\in \F_p$ such that $\|(t\Bw)'\|_2 < \kappa$ then we have 
				$$\rho(\Bw) \le \exp(-\kappa^2/2).$$
				\item Otherwise, assume that $1\le \|\Bw'\|_2 \le \kappa$ and  $\Bw$ satisfies \eqref{eqn:sparse}, with $\kappa\le C\sqrt{n}$ and
				$$\ULCD_{\gamma, \kappa}(\Bw') \ge L.$$ 
				Then 
				$$\rho(\Bw) \le  O\Big(\frac{1}{ L\|\Bw'\|_2} +\exp(-\Theta(\kappa^2)) \Big),$$
				where the implied constants depend on $C, \gamma, c_{nsp}$, and where $X=(x_1,\dots, x_n)$ and $x_i$ are iid Bernoulli in the concentration definition $\rho(w)$.
			\end{enumerate}
		\end{theorem}

		\begin{corollary}
			Assume that there exists a quantity $\rho \gtrsim \exp(-\Theta(\kappa^2))$ such that $\rho(\Bw) \ge \rho$, then there exists a dilation $t\Bw$ of $\Bw$, where $t \in \F_p$ non-zero, so that with $\Bw'=t\Bw$ we have $\|\Bw'\|_2 < \kappa$ and there exists $L=L(\Bw) \ge 1$ such that 
			$$  L = O\Big( \frac{\rho^{-1}}{\|\Bw'\|_2} \Big)$$
			and 
			$$\dist(L\Bw', \Z^n) \le \kappa.$$ 
		\end{corollary}
		We next deduce another elementary but useful result, which will be used later on in the applications.
		\begin{corollary}\label{cor:smallp} Assume that $\Bw$ has at least $m$ non-zero coordinates, and $p<\sqrt{m}$. We then have 
			$$\rho(\Bw) \le \exp(-m/2p^2).$$
		\end{corollary}
		\begin{proof} As $t\Bw'$ has at least $m$ non-zero coordinates for any non-zero $t$, we have that 
			$$\|t\Bw'\|_2 \ge \sqrt{m (1/p)^2} = \sqrt{m}/p,$$ and we are in scenario (1) of Theorem \ref{theorem:inverse:modp}.
		\end{proof}

		We now present a proof of Theorem \ref{theorem:inverse:modp}.
		\begin{proof}(of Theorem \ref{theorem:inverse:modp}) Write $e_p(x) = e^{2\pi i x/p}$, then for any $r \in \F_p$ we have 
			$$\P(X \cdot \Bw =r) -\frac{1}{p} = \frac{1}{p} \sum_{t \in \F_p, t\neq 0} \prod_{l=1}^n \E e_p( x_l w_l t) e_p(-t r).$$
			So
			\begin{align*}|\P(X \cdot \Bw=r) -\frac{1}{p}| &\le  \frac{1}{p}\sum_{t \in \F_p, t\neq 0} |\prod_{l=1}^n \E e_p(x_l w_l t)|\\ 
			&=\frac{1}{p}\sum_{t \in \F_p, t\neq 0} \prod_{l=1}^n |\cos(2 \pi w_l t/p)|\\
			&=\frac{1}{p}\sum_{t \in \F_p, t\neq 0} \prod_{l=1}^n |\cos( \pi w_l t/p)|\\
			&\le  \frac{1}{p}\sum_{t \in \F_p, t\neq 0}  e^{- 2\sum_{l=1}^n \|\frac{w_lt}{p}\|^2}\\
			&\le  \frac{1}{p}\sum_{t \in \F_p, t\neq 0}  e^{- 2\sum_{l=1}^n \| t w_l'\|_{\R/\Z}^2},
			\end{align*}
			where we used the fact that   $|\sin  \pi z | \ge 2 \|z\|_{\R/\Z}$ for any $z\in \R$, where $\|z\|_{\R/\Z}$ is the distance of $z$ to the nearest integer, and that
			\begin{equation*} \label{eqn:fourier3-1}| \cos \frac{2 \pi x}{p}|  \le  1- \frac{1}{2} \sin^2 \frac{2\pi x}{p}  \le 1 -2 \|\frac{2x}{p} \|^2. \end{equation*}
			From here, (1) follows as $\| t w_l'\|_{\R/\Z}= \|(tw)_l'\|_{\R/\Z}.$
			
			We are now in the assumption of (2). For each integer $m$, let $T(m,p/2)$ be the (level) set of $t\in \F_p$ corresponding to $m$,
			$$T(m,p/2) =\Big\{ t \in \F_p: \|t\Bw'\|_{\R/\Z}^2:=\sum_{l=1}^n \| t w_l'\|_{\R/Z}^2 \le m\Big \} = \Big\{ t \in \I_p: \|t\Bw'\|_{\R/\Z}^2 \le m\Big \}.$$
			By the non-sparsity of $\Bw$, we can show that $T(c_{nsp} n/64, p/2)$ is not all of $\F_p$ (we can show this by using the fact that if $w_i \neq 0$ then $\sum_{t \in \F_p} \|\frac{t w_i}{p}\|_{\R/\Z} = \sum_{t \in \F_p} \|\frac{t}{p}\|_{\R/\Z}$). Thus
			$$|T(c_{nsp} n/64, p/2)| \le p-1.$$ 
			Our next claim shows that the level sets consist of well separated intervals.

			\begin{claim}[spacing of the level sets]\label{claim:separation} Assume that $m< \kappa^2/4<c_{nsp} n/64$ and $s_1<s_2 \in T(m,p/2)$ and 
				$$|s_2-s_1| \ge \gamma^{-1}\kappa/\|w'\|_2.$$ 
				Then
				$$s_2 -s_1 \ge L.$$
				Consequently we have $$|T(m,p/2)| \le  \frac{p \lceil \gamma^{-1} \kappa/\|\Bw'\|_2 \rceil}{L}.$$
				
			\end{claim}
			\begin{proof}(of Claim \ref{claim:separation}) Assume that $t_1,t_2 \in T(m,p/2)$, then by the triangle inequality,
				$$\|(t_1-t_2)\Bw'\|_{\R/\Z} \le \|t_1 \Bw'\|_{\R/\Z}+\|t_2 \Bw'\|_{\R/\Z}  \le 2\sqrt{m} <\kappa.$$
				Thus
				$$\dist((t_2 -t_1)\Bw', \Z^n)  = \|(t_2-t_1)\Bw'\|_{\R/\Z} < \kappa \le \gamma (t_2-t_1) \|\Bw'\|_2,$$
				where in the last estimate we used $|t_2-t_1| >\gamma^{-1}\kappa/\|\Bw'\|_2$. Thus by the definition of $\ULCD_{\gamma,\kappa}$ we must have 
				$$t_2-t_1 \ge L.$$
			\end{proof}
			Next, we will need a Cauchy-Davenport-type bound on size of sumsets in $\F_p$. Observe from the Cauchy-Schwarz inequality, $k(\|x_1\|_{\R\Z}^2+\dots+ \|x_k\|_{\R/\Z}^2) \ge \|x_1+\dots+x_k\|_{\R/\Z}^2$, and so
			$$k|T(m, p/2)| \le |T(k^2m, p/2)|,$$ 
			where we view these sets as subsets of $\F_p$. Hence, by Cauchy-Davenport's inequality in $\F_p$ (\cite{TVbook}) we have that 
			$$|T(k^2m)| \ge |kT(m)| \ge k|T(m)| - (k-1).$$
			Thus for all $m\le \min\{c_{nsp} \kappa^2/64, c_{nsp} n/64 \}$, by choosing $k= \lfloor \sqrt{\min \{c_{nsp}\kappa^2/64, c_{nsp} n/64\}/m} \rfloor$ we have
			\begin{equation}\label{eqn:mn}
			|T(m,p/2)|-1 \le k^{-1} (|T(c_{nsp} n/64, p/2)|-1) \le \sqrt{\frac{m}{O_{c_{nsp}, C} (\kappa^2)}}  p,
			\end{equation}
			where we used $|T(c_{nsp} n/64, p/2)|-1  < p$.

			We deduce
			\begin{align*}
			|\P(X \cdot \Bw=r) -\frac{1}{p}| &\le \frac{1}{p}\sum_{t \in \F_p, t\neq 0}  e^{- 2\sum_{l=1}^n \| t \Bw'\|^2}\\ 
			&\le \frac{1}{p} \Big(\frac{p \lceil \gamma^{-1} \kappa/\|\Bw'\|_2 \rceil}{L \sqrt{O_{c_{nsp},C}(\kappa^2)}} \sum_{m\le  \min\{c_{nsp} \kappa^2/64, c_{nsp} n/64 \} } \sqrt{m}e^{-m/2}  +   p \sum_{m>  \min\{c_{nsp} \kappa^2/64, c_{nsp} n/64 \} } e^{-m/2}\Big) \\
			& =O_{c_{nsp, C, \gamma}}\Big(\frac{1}{L\|\Bw'\|_2} +\exp(-\Theta_{C ,c_{nsp}}(\kappa^2))\Big),
			\end{align*}
			completing the theorem proof.
		\end{proof}
		
		\begin{remark}\label{remark:inverse:modp} Under the assumption of (2) of Theorem \ref{theorem:inverse:modp}, we have actually shown a stronger estimate that 
			$$\frac{1}{p}\sum_{t \in \F_p, t\neq 0}  e^{- 2\sum_{l=1}^n \| t \Bw'\|^2} \le O\Big(\frac{\lceil c_{nsp}^{-1/2}\gamma^{-1} \kappa/\|\Bw'\|_2 \rceil}{ \kappa L} +\exp(-c_{nsp} \kappa^2/64)\Big).$$
		\end{remark}
		
		We note that Theorem \ref{theorem:LO} can be deduced from Theorem \ref{theorem:inverse:modp} by setting $\kappa=C\sqrt{n}$ with sufficiently large $C$; for this there is a dilation $\Bw=t\Bw$ with $\|\Bw'\|_2$ of order $\sqrt{n}$ but $\|\Bw'\|_2 < \kappa$. We then just apply (2) of Theorem \ref{theorem:inverse:modp}, noting that $L\ge 1$.
		
		\begin{remark}\label{remark:inverse:modp:alpha}  When $x_i$ are iid copies of an $\al$-balanced random variable, then by Equation \eqref{eqn:fourier4}, by convexity, and by the fact that $1-\al/2 \le \sum_{j=1}^l \beta_j \le 1$ we have
			\begin{align*} 
			\rho  &\le  \frac{1}{p} \sum_{x \in \F_p, x \neq 0} \exp( - 2 \sum_{i=1}^n \sum_{j=1}^l \beta_j \|\frac{x t_j w_i}{p} \|^2) \\
			& \le  \frac{1}{p} \sum_{x \in \F_p, x \neq 0} \exp( - 2 \sum_{j=1}^l \frac{\beta_j}{\sum_{j=1}^l  \beta_j} (\sum_{j=1}^l \beta_j) \sum_{i=1}^n  \|\frac{x t_j w_i}{p} \|^2) \\
			&\le  \frac{1}{p} \sum_{x \in \F_p, x \neq 0} \sum_{j=1}^l   \frac{\beta_j}{\sum_{j=1}^l  \beta_j}  \exp( - 2 ( \sum_{j=1}^l \beta_j )\sum_{i=1}^n\|\frac{x t_j w_i}{p} \|^2).\\
			&\le  \sum_{j=1}^l   \frac{\beta_j}{\sum_{j=1}^l  \beta_j}   \frac{1}{p} \sum_{x \in \F_p, x \neq 0} \exp( - 2 (1-\al/2)\sum_{i=1}^n\|\frac{x t_j w_i}{p} \|^2)\\
			&= \frac{1}{p} \sum_{x \in \F_p, x \neq 0} \exp( - 2 (1-\al/2)\sum_{i=1}^n\|\frac{x w_i}{p} \|^2).
			\end{align*}
			It thus boils down to study the bounds for concentration probability of $\Bw$, for which we have done in the proof of Theorem \ref{theorem:inverse:modp}. 
		\end{remark}
		
		
		\subsection{Some properties of ULCD}  Roughly speaking, our next result is similar to Theorem \ref{theorem:inverse:modp}, but instead of working with the concentration event $X \cdot \Bw =r$  we are working with a coarser event that $X \cdot \Bw$ belongs to an arc in $\F_p$. We find it more convenient to write in mod 1 as follows.

		\begin{theorem}[anti-concentration modulo one]\label{theorem:mod1}  Assume that $0<a_1,\dots, a_n<1$. Assume that 
			$$\ULCD_{\gamma, \kappa}((a_1,\dots,a_n)) = L \ge 1.$$
			Then for any 
			$$\eps > 2/L$$ 
			we have
			$$\P(\|\sum_i \xi_i a_i\|_{\R/\Z} \le \eps) \le O\Big(\eps + \log(1/\eps) [\exp(-\kappa^2) + \exp(-4 \gamma^2  \sum_i a_i^2)]\Big)$$
			where $\xi_i$ are iid Bernoulli.
		\end{theorem}
		
		Note that we need this result because at some point we need to pass to characteristic zero, and take distance to $\Z$. A key difference of this bound compared to the classical small ball estimate (say studied in \cite{RV}) is that we are looking at the balls modulo one, rather than with respect to the whole real line.

		
		\begin{proof}(of Theorem \ref{theorem:mod1}) Let $\mu$ be the distribution of $\sum_i a_i \xi_i$ modulo one, where we can write $\mu = \mu_1\ast \dots \ast \mu_n$, and where $\mu_i(a_i)=\mu(-a_i) =1/2$. Let $L_0 = \lfloor 1/\eps \rfloor$. We use the Erd\H{o}s-Tur\'an inequality, 
			$$\mu([-\eps,\eps]) = \P(\|\sum_i a_i \xi_i \|_{\R/\Z} \le \eps) \le \eps + \frac{1}{L_0} + \sum_{k=1}^{L_0} \frac{|\widehat{\mu}(k)|}{k},$$
			As $\xi_i$ are iid Bernoulli, bounding the cosine as in the proof of Theorem \ref{theorem:inverse:modp} we have
			$$|\widehat{\mu}(k)| = |\int_{0}^1 e^{2\pi \sqrt{-1} k \theta} d\mu(\theta)|  = |\prod_i \int_{0}^1 e^{2\pi \sqrt{-1} k \theta} d\mu_i(\theta)|  \le \exp(- \sum_i \| 2 a_i k\|_{\R/\Z}^2).$$
			Now by definition of $\ULCD_{\gamma,\kappa}$, as $L _0 \le 1/\eps < \ULCD_{\gamma, \kappa}((a_1,\dots,a_n))/2$, for any $k\le L_0$ we have 
			$$\sum_i \| 2 a_i k\|_{\R/\Z}^2 \ge \min\{\kappa^2, 4\gamma^2 k^2 \sum_i a_i^2\},$$ and so
			$$\exp(- \sum_i \| 2a_i k\|_{\R/\Z}^2) \le  \exp(-\kappa^2)+\exp(-4\gamma^2 k^2 \sum_i a_i^2) \le  \exp(-\kappa^2)+\exp(-4\gamma^2\sum_i a_i^2).$$
			Summing over all $k\le L_0$ we have
			\begin{align*}
			\P(\|\sum_i \xi_i a_i\|_{\R/\Z} \le \eps) \le  O(\eps) + \log(1/\eps) [\exp(-\kappa^2) + \exp(-4 \gamma^2  \sum_i a_i^2)] ,
			\end{align*}
			as desired.
		\end{proof}
		It is remarked that the bound above depends on $\|\Ba\|_2$, which becomes almost meaningless if $\|\Ba\|_2$ is small, say of order $O(1)$. To avoid this situation, we will need to consider vectors $\Bw'$ that have large size and large $\ULCD_{\gamma,\kappa}(\Bw')$ at the same time.
		
		Our next result roughly says that a non-sparse vector cannot have very small $\ULCD$, at least with respect to $\F_p$ with not too large and not too small $p$. To be more precise, we have the following.
		
		\begin{remark} As we will be working with vectors $\Bw$ satisfying \eqref{eqn:sparse}, we easily see that for any $t \neq 0$ in $\F_p$
			$$\|(t\Bw)'\|^2 \ge c n (\frac{1}{p})^2.$$ 
			Notice that this quantity is larger than $\kappa^2$ if $p\le n^{1/2}/\kappa$, and in this case the first part of \ref{theorem:inverse:modp} holds, and hence automatically 
			$$\rho(\Bw) \le \exp(-c\kappa^2).$$ 
		\end{remark}
		
		As such, in what follows we will be working with 
		$$p \gtrsim n^{1/2}/\kappa.$$
		
		\begin{lemma}[LCD and size in fields of small order]\label{lemma:largeLCD} Assume that  $\kappa =n^c$ for a positive constant $c<1/16$. Assume that $p$ is a prime smaller than $\exp(c\kappa^2)$, and $w \in \F_p^n$ is a vector satisfying \eqref{eqn:sparse} and such that  $\rho(\Bw) \ge 2\exp(-\kappa^2/2)$. Then there exists $t\in \F_0$ so that with $\Bw=t\Bw$ we have $\|\Bw'\|_2$ has order $\kappa$ and either  $\ULCD_{\gamma,\kappa}(\Bw') = p$ (in which case we can apply \eqref{eqn:p}) or else 
			$$\ULCD_{\gamma,\kappa}(\Bw') \ge \kappa^{5/4-c}.$$
		\end{lemma}
		
		We remark that this result is perhaps the most important one in our treatment, as it allows us to assume that the $\ULCD$ to be sufficiently large to make sense of the bounds. In characteristic zero, this bound is straightforward if the vector is incompressible (being far from sparse vectors).
		
		Before proving this lemma, we first need the following simple statement.
		
		\begin{claim}\label{claim:increasing:iid} Assume that  $\Bw \in \F_p^n$ is a non-zero vector satisfying \eqref{eqn:sparse} and such that  $\rho(\Bw) \ge 2\exp(-\kappa^2/2)$, with $\kappa =o(\sqrt{n})$. Then there exists $t\in \F_0 = \F_p \setminus 0$ so that with $\Bw=t\Bw$ we have 
			$$\kappa/2 \le \|\Bw'\|_2 < \kappa.$$ 
		\end{claim}
		\begin{proof}(of Claim \ref{claim:increasing:iid}) As $\Bw$ satisfies \eqref{eqn:sparse} and $\rho(\Bw) \ge 2\exp(-\kappa^2/2)$, (1) of Theorem \ref{theorem:inverse:modp} does not apply, and so there is a fiber $\Bw=t\Bw$ such that $\|\Bw'\|_2 < \kappa$. If $\|\Bw'\|_2 \ge \kappa/2$ the we would be done.  Otherwise we just consider the sequence $\Bw, 2\Bw, 3\Bw$, etc. By the triangle inequality (where we recall that $(t\Bw)' = \frac{1}{p}(tw_1(\mod p),\dots, tw_n (\mod p))$) we have
			$$\|((k+1)\Bw)'\|_{2} \le \|(k\Bw)'\|_{2} + \|\Bw'\|_2.$$
			On the other hand, by \eqref{eqn:sparse} $\sum_{k\in \F_p}\|(k\Bw)'\|_{2}$ has order $\sqrt{n}p$, so there must exist a smallest $k_0\ge 2$ such that $\|((k_0+1)\Bw)'\|_2 \ge \kappa$. It then follows that $\kappa/2 \le \|(k_0\Bw)'\|_2 <\kappa$.
		\end{proof}

		\begin{proof}(of Lemma \ref{lemma:largeLCD}) Assume that we are not in the first case, and also assume to the contrary that we are not in the second case either. We will iterate the following process, which will then result in a contradiction. Set 
			$$\beta = 1/4 - c.$$
			We start from any $\Bu_1=\Bw'= (w_1/p,\dots, w_n/p)$ in the fiber $t\Bw$ of $\Bw$ with $\kappa/2 \le \|\Bu_1\|_2 < \kappa $. 
			
			{\bf Step 1}: Let $D_1=\ULCD(\Bu_1)$, then $2\le D_1 \le \min\{\kappa^{1+\beta},p-1\}$. Let $\Bu_1' = D_1 \Bu_1 (=(D_1 \Bw)')$, then we have 
			$$\|\Bu_1'\|_2 \le \kappa.$$ 
			
			{\bf Step 2}: If this vector has norm smaller than $\kappa/2$, then we use Claim \ref{claim:increasing:iid} to dilate appropriately by $C_1 \ge 2$ so that $\kappa/2 \le \|C_1 \Bu_1'\|_2 \le \kappa$, and set $$\Bu_2=C_1 \Bu_1' (= (C_1 D_1 \Bw)'.)$$ 
			
			We then return to Step 1 and iterate the process, note that while the $D_i$ are bounded by $\kappa^{1+\eps}$, we don't have such a bound for the $C_i$.
			
			Now for each $1\le t\le p-1$ we can always write 
			$$t = D_1 ( C_1 (D_2 (\dots)+r_2)+ s_1) + r_1,$$
			where $r_1 <D_1, s_1<C_1, r_2<D_2, s_2 <C_2,\dots$. Indeed, to verify this we first divide $t$ by $D_1$ and get a remainder $r_1$; we then divide the quotient by $C_1$ to get a remainder $s_1$, and then divide the new quotient by $D_2$, etc until the last step. Now as $t\le p-1 \le 2^{\kappa^2}$ (this is where we require $p$ to be small), and as $C_i,D_i\ge 2$, we must stop the division process after $\kappa^2$ steps.
			
			Next we we analyze the norm of $\|t\Bu_1\|_2 =\|(t\Bw)'\|_2$. We write, with $t_1=t$ 
			$$t_1\Bu_1 = D_1 ( C_1 (D_2 (\dots)+r_2)+ s_1) \Bu_1 + r_1 \Bu_1=   ( C_1 (D_2 (\dots)+r_2)+ s_1) D_1 \Bu_1 + r_1 \Bu_1:=t_1' \Bu_1' + r_1 \Bu_1.$$
			Thus by the triangle inequality, and as $r_1 \le D_1-1 < \kappa^{1+\beta}$ we have 
			$$\|t_1\Bu_1\|_2= \|t\Bu_1\|_2 \le \|t_1' \Bu_1'\|_2 + \kappa^{2+\beta}.$$
			We next consider 
			$$t_1'\Bu_1' = ( C_1 (D_2 (\dots)+r_2)+ s_1) \Bu_1' =  C_1 (D_2 (\dots)+r_2) \Bu_1' + s_1 \Bu_1'= (D_2 (\dots)+r_2) \Bu_2 + s_1 \Bu_1':=  t_2 \Bu_2 + s_1 \Bu_1'.$$
			By the triangle inequality
			$$\|t_1' \Bu_1'\|_2 \le  \|t_2 \Bu_2\|_2 + \|s_1 \Bu_1'\|_2 \le  \|t_2 \Bu_2\|_2 + \kappa,$$
			where in the last estimate we used the fact that $0\le s_1\le C_1-1$ and $C_1$ is the largest integer so that $\kappa/2 \le \|C_1 \Bu_1'\|_2 \le \kappa$ (where we recall that $\|\Bu_1'\|_2 \le \kappa$, and the role of $C_1$ was only to dilate this vector if its norm was much smaller than this, as in the proof of Claim \ref{claim:increasing:iid}). The analysis for $\|t_2 \Bu_2\|_2$ and other terms can be done similarly. 
			
			Adding all the bounds, we hence obtain
			$$\|t\Bu_1\|_2 \le \kappa^2 \times \kappa^{2+\be} = \kappa^{4+\be}.$$
			Now as this is true for all $t\in \F_p$, we thus have
			$$\sum_{t\in \F_p} \|t\Bw'\|_2^2 = O( p \kappa^{4+\be}).$$
			On the other hand, by \eqref{eqn:sparse}, as $w'$ has at least $c_{nsp} n$ non-zero entries, the left hand side can be shown to be at least $c_{nsp}np/64$, which is a contradiction if $\kappa \le n^c = n^{1/4 -\beta}$.
		\end{proof}
		
		With the same proof, we record the following corollary which will be used later.
		
		\begin{cor}[ULCD cannot be small]\label{cor:largeLCD} Assume that $p \le \exp(c\kappa^2)$ and that $\kappa =n^c$ for $c<1/16$. Assume that $\Bw\in \F_p^m$, for $m \ge \kappa^{4+2 \eps}$, and $\Bw$ has at least $\kappa^{4+2 (1/4 - c)}$ non-zero components. Then we either have either $\|(t\Bw)'\|_2 > \kappa$ for all $t$, or there exists such $\Bw'=(t\Bw)'$ such that $\kappa/2 \le \|\Bw'\|_2 < \kappa$ and that 
			$$\ULCD_{\gamma,\kappa}(\Bw') \ge \kappa^{5/4 -c}.$$
		\end{cor}

		\section{Structures of vectors in $\F_p^n$: a combinatorial approach}\label{section:combinatorial} 
		
		Now we present our third characterization. Let $\mu$ be an $\al$-balanced distribution in $\F_p$. For simplicity, we again assume $\mu$ to be Bernoulli $\pm 1$, the general $\al$-balanced case can be treated almost identically as in the previous two sections. Our goal here is the following.

		\begin{theorem}[Combinatorial structure, characterization III]\label{theorem:Halasz} Let $k\ge 1$ be an integer. Let $f: \Z^+ \to \Z^+$ be any function such that $f(x) \le x/100$.
			For any non-zero vector $\Bw= (w_1, \cdots, w_n) \in \F_p^n$ we have 
			$$\rho(\Bw) = \sup_{a \in \Z/p\Z} |\P(\mu_1w_1 + \cdots + \mu_nw_n = a)-\frac{1}{p}| \le \frac{R_k(\Bw)}{ 2^{2k} n^{2k} \sqrt{f(|\supp(\Bw)|)}} + e^{-f(|\supp(\Bw)|)/2}.$$ 
			where $\mu_1, \cdots, \mu_n$ are independent and identically distributed copies of $\mu$, and where $R_k(\Bw)$ is the number of solutions to $ \pm w_{i_1} \pm \dots \pm w_{i_{2k}} =0 (\mod p)$, and $k\le n/f(|\supp(\Bw)|)$.
			
		\end{theorem}
		
		Our approach here is somewhat similar to \cite{FJLS}, which in turn follows the original approach of Hal\'asz in \cite{Halasz}). However, the key difference here is that we are estimating the deviation of $\sup_{a \in \Z/p\Z} |\P(\mu_1w_1 + \cdots + \mu_nw_n = a)$ from $1/p$ rather then giving an upper bound for $\sup_{a \in \Z/p\Z} |\P(\mu_1w_1 + \cdots + \mu_nw_n = a)$ as in \cite{FJLS}.

		\begin{proof}(of Theorem \ref{theorem:Halasz})  We follow the proof of Theorem \ref{theorem:inverse:modp}  until Equation \eqref{eqn:mn} that
			$$|T(m,p/2)|-1 \le k^{-1} (|T(|\supp(\Bw)|/64, p/2)|-1),$$
			where $k= \Theta(\sqrt{f(|\supp(\Bw)|)/m})$.
			
			Denote $T':=\{l \in \F_p, |\sum_{j=1}^{n} \cos(2\pi l w_j/p)| \ge n - 100 f(|\supp(\Bw)|)\}$. Then we see that 
			$$T(|\supp(\Bw)|/64, p/2)| \le |T'|.$$
			By Markov's inequality,
			$$|T'| \le \frac{1}{ (n- 100 f(|\supp(\Bw)|))^{2k}} \sum_{l\in T'} |\sum_{j=1}^{n} \cos (2\pi l w_j/p)|^{2k}.$$ 
			By expanding out the RHS and summing over $l\in \F_p$ instead, we can bound the RHS from above by 
			$$T(|\supp(\Bw)|/64, p/2)| \le |T'|  \le \frac{\sqrt{2} p R_k(\Bw)}{2^{2k} n^{2k}}.$$
			The rest can be completed as in Theorem \ref{theorem:inverse:modp}:
			
			\begin{align*}
			|\P(X \cdot \Bw=r) -\frac{1}{p}| &\le \frac{1}{p}\sum_{t \in \F_p, t\neq 0}  e^{- 2\sum_{l=1}^n \| t \Bw'\|^2}\\ 
			&\le \frac{1}{p} \Big(\sum_{m\le f(|\supp(\Bw)|) } \sqrt{\frac{m}{f(|\supp(\Bw)|)}} \frac{\sqrt{2} p R_k(\Bw)}{2^{2k} n^{2k}} e^{-m} +   p \sum_{m>  f(|\supp(\Bw)|) } e^{-m}\Big) \\
			&\le \frac{R_k(\Bw)}{ 2^{2k} n^{2k} \sqrt{f(|\supp(\Bw)|)}} + e^{-f(|\supp(\Bw)|)/2},
			\end{align*}
			as claimed.
		\end{proof}

		\section{Non-structures of normal vectors}\label{section:normal:non-structure}
		In this section we use the three characterizations above to establish Theorem \ref{thm:structure:iid}. First, it is easy to show that normal vectors are non-sparse with high probability.
		
		\begin{lemma}\label{nonsparselem} There exists an absolute constant $c_{nsp}>0$ such that with probability $1-\exp(-\Theta(n))$ any normal vector $w$ of $span(X_1,\dots, X_{n-1})$ satisfies \eqref{eqn:sparse}.
		\end{lemma}
		\begin{proof} This follows from Odylzko's lemma, see for instance \cite{NgW1,NgP}. 
		\end{proof}
		
		Now we use the results from Sections \ref{section:GAP}, \ref{section:ULCD}, and \ref{section:combinatorial} to show that the normal vectors cannot have any structure. In our first proposition, we use the structure from Section \ref{section:GAP}.
		\begin{proposition}[Normal vectors cannot have additive structures]\label{prop:GAPstructure:iid} Let $C>0$. Let $X_1,\dots, X_{n-d}$ be the first $n-d$ columns of a matrix $M$ whose entries are iid copies of a $\a$-balanced random variable, where $d\le cn$ for some sufficiently small constant $c$. Let $\Bw$ be any non-zero vector that is orthogonal to $X_1,\dots, X_{n-d}$. Then with probability at least $1 -\exp(-\Theta(n))$, the vector $\Bw$ cannot have structure as in the conclusion of Theorem \ref{theorem:ILO}. In particular, we have 
			$$\rho(\Bw) =O(n^{-C})$$
			where the implied constant depends on $C$.
		\end{proposition}


		\begin{proof}(of  Proposition \ref{prop:GAPstructure:iid}) First of all, from Lemma \ref{nonsparselem}, with a loss of $\exp(-\Theta(n))$ in probability we can assume that $\Bw$ is not sparse. Assume that 
			$$\rho=\rho(\Bw) \asymp n^{-C},$$
			where $C\ge 1/2$. Also, by Corollary \ref{cor:smallp}, it suffices to assume $p\gtrsim n^{1/2}.$
			
			For convenience, let $p'=\min\{p, \rho^{-1}\}$, and so $p'\gtrsim n^{1/2}$. Then by Theorem \ref{theorem:ILO}, we have a generalized arithmetic progression $P$ of rank $O(1)$ in $\F_p$ and of size $O(1+\min\{\rho^{-1}/n^\eps, p/n^\eps\})=O(p'/n^\eps)$ that contains all but $n^{2\eps}$ entries of $\Bw$. Note that the number of ways to choose such a $P$ is bounded by
			$$p^{O(1)} \rho^{-O(1)}.$$
			Given $P$, the number of vectors $\Bw$ whose $n-n^{2\eps}$ components are from $P$ is at most 
			$$\binom{n}{n^{2\eps}} |P|^{n-n^{2\eps}} p^{n^{2\eps}} \le 2^n (p'/n^\eps)^{n-n^{2\eps}} p^{n^{2\eps}} \le 8^n (p'/n^\eps)^{n},$$
			provided that $p\le \exp(n^{1-2\eps})$.
			
			Given $\Bw$ for which $\rho(\Bw) \ge \rho$, the probability that $\Bw$ is orthogonal to $X_1,\dots, X_{n-d}$ is bounded by 
			$$8^n (p'/n^\eps)^{n}  (\frac{1}{p} + \rho)^{n-d} \le 8^n (p'/n^\eps)^{n}  (\frac{2}{p'})^{n-d}   \le n^{-\eps n/2},$$
			provided that $d\le c n$ for some small positive constant $c$. 
			
			Taking union bound over only $p^{O(1)} \rho^{-O(1)}$ choices of $P$, we obtain the claim. 
		\end{proof}
		
		We next use the result from Section \ref{section:ULCD} to show that the random normal vector does not have small $\ULCD_{\gamma,\kappa}$.
		
		\begin{prop}[Normal vectors cannot  have small ULCD]\label{prop:structure:iid} Assume that $p \le \exp(c\kappa^2)$ and that $\kappa =n^c$ for $c<1/16$. Let $X_1,\dots, X_{n-d}$ be the first $n-d$ columns of a random $(-1,1)$ Bernoulli matrix, where $d\le n^c$. Let $\Bw$ be any non-zero vector that is orthogonal to $X_1,\dots, X_{n-d}$. Then with probability at least $1 -\exp(-\Theta(n))$, we have 
			$$\ULCD_{\kappa}(\Bw) \ge \exp(c'\kappa^2)$$
			with some $c'$ depending on $c$ and $\gamma$. In particular, Theorem \ref{thm:structure:iid} holds.
		\end{prop}
		
		We remark that in the above theorem we assume $M$ to be a $(-1,1)$ Bernoulli matrix. Our treatment also works for other integral matrices \footnote{Here the random entries of $M$ take value in $\Z$, although in our results we view $M$ as a matrix of entries from $\Z/p\Z$ (or $\F_p$) via the natural map $\Z \to \Z/p\Z$.} with $\|M\|_2 = O(\sqrt{n})$ but it does not seem to extend to $\al$-balanced ensemble as in Proposition \ref{prop:GAPstructure:iid} (although Theorem \ref{theorem:inverse:modp} holds for this setting). The main reason is that at some point in the proof we pass to a net of vectors in $\R^n$, and then under the action of $M$ the size of this net will blow up if $M$ has large norm, see \eqref{M:norm}. 
		
		
		\begin{proof}(of Proposition \ref{prop:structure:iid}) We will show that with high probability, there does not exist $\Bw$ in the fiber of $t\Bw$ such that $\kappa/2 \le \|\Bw'\|_2<\kappa$ and that 
			$$\kappa^{1+(1/4-c)} < \ULCD_{\gamma,\kappa}(\Bw') \le \exp(c' \kappa^2/2)/\kappa.$$
			To do this, we divide this range into $O(\kappa^2)$ dyadic intervals $(D_i,D_{i+1}=2D_i)$. For $D=D_i$, let 
			$$S_D = \Big\{\Bw' = (w_1/p,\dots, w_n/p): \kappa/2 \le \|\Bw'\|_2 < \kappa \wedge  D \le \ULCD_{\gamma,\kappa}(\Bw') \le 2D\Big\}.$$
			
			\begin{lemma}[Size of the approximating net]\label{lemma:net:iid} Let $c_0>0$ be given sufficiently small compared to $c$ (where $\kappa=n^c$). $S_D$ accepts a $O(\kappa/D)$-net $\CN$ of size $D (C \kappa D/\sqrt{n})^n$ if $\kappa D \ge c_0 \sqrt{n}$ and of size $D n^{c_0 n}$ if $\kappa D < c_0 \sqrt{n}$ and such that $\CN \subset S_{D}$. 
			\end{lemma}
			
			Before proving this result by following \cite{RV}, let use introduce a fact that will be useful to our nets.
			
			\begin{fact}\label{fact:passingnet} Assume that $\CS$ accepts a $\delta$-net $\CU$ of size $|\CN|$, then $S$ also accepts a $2\delta$-net $\CU'$ such that $\CU' \subset \CS$ and which has size at most $|\CN|$. 
			\end{fact}
			\begin{proof} By throwing away vectors from $\CU$ if needed, we assume that each $u\in \CU$ $\delta$-approximates at least one vector $s'$ from $\CS$. Let $\CN'$ be a collection of such $s'$ (we choose an arbitrary $s'$ from $\CS$ that is $\delta$-approximated by any $u$.) Thus  $\CN' \subset \CS \mbox{ and } |\CN'| \le |\CN|$.
				Now for any $s\in \CS$, there exists $u\in \CU$ such that $\|u-s\|_2 \le \delta$, and also by definition there also exists $s'\in \CU'$ such that $\|u-s'\|_2 \le \delta$. Thus we have $\|s-s'\|_2 \le 2\delta$, so $\CU'$ is a $2\delta$-net of $\CS$.   
			\end{proof}

			\begin{proof}[Proof of Lemma \ref{lemma:net:iid}] By taking union bound over a small number of choices (at most $O(\kappa \times (D/\kappa)) =O(D)$ choices) we assume that for some $T \in\kappa/D \cdot \Z $ we have 
				$$T- \kappa/D \le \|\Bw'\|_2 \le T +\kappa/D.$$
				By definition, as $\|L \Bw'\|_{\R/\Z}\le \kappa$ and $D\le L \le 2D$, there exists $\Bp\in \Z^{n}$ such that
				$$\left\|L \Bw' -\Bp\right\|_2 \le \kappa.$$ 
				
				This implies that 
				$$\left\|\Bw' -\frac{\Bp}{L}\right\|_2 \le \frac{\kappa}{L} \le \frac{\kappa}{D},$$
				and hence
				$$ T  - 2\frac{\kappa}{D}  \le \frac{\|\Bp\|_2}{L} \le T + 2\frac{\kappa}{D} .$$
				
				Thus
				\begin{align*}
				\|\Bw' - T \frac{\Bp}{\|\Bp\|_2}\|_2 &\le \|\Bw' -\frac{\Bp}{L}\|_2 +  \|T \frac{\Bp}{\|\Bp\|_2} -\frac{\Bp}{L}\|_2\\ 
				&\le  \frac{\kappa}{D} +  \|\Bp\|_2 |\frac{T}{\|\Bp\|_2} -\frac{1}{L}| \\
				& \le 3\frac{\kappa}{D}.
				\end{align*}

				Now as $\|\Bw'\|_2 < T+ 2\kappa/L$, we also have $\|\Bp/L\|_2 \le T + 3\kappa/L$ and so 
				$$\|\Bp\|_2 \le  2D T + 3\kappa \le 2D(\kappa + \kappa/D) + 2\kappa < 4D \kappa.$$
				Let $\CN$ be the collection of vectors $T \frac{\Bp}{\|\Bp\|_2}$, where $T$ ranges over $O(D)$ choices in the set $\kappa/D \cdot \Z$, and $\Bp$ ranges over all integer vectors in $\Z^n$ satisfying $\|\Bp\|_2 \le 4 D \kappa$.

				Now we bound the size of $\CN$ basing on the magnitude of $\kappa D$.
				
				{\bf Case 1.} If $\kappa D \ge c_0 \sqrt{n}$, then the number of integral vectors $\Bp$ of norm at most $3\kappa D$ is known to be bounded by $(C\kappa D/\sqrt{n})^n$, and so
				$$|\CN| \le D (C\kappa D/\sqrt{n})^n.$$
				
				{\bf Case 2.} If $\kappa D \le c_0 \sqrt{n}$, where $c_0$ is sufficiently small, then all but $O((\kappa D)^2)$ entries of $\Bp$ are zero. So the number of such vectors $\Bp$ is bounded by $\binom{n}{(\kappa D)^2} (O(1))^{(\kappa D)^2}$, and so
				$$|\CN| \le D \binom{n}{(\kappa D)^2} C^{(\kappa D)^2} \le D n^{c_0 n}.$$
				Finally, we can always assume $\CN$ to consist of vectors from $S_D$ by using Fact \ref{fact:passingnet}.
			\end{proof}
			
			Now we use the obtained net to show that normal vectors in iid matrices cannot have small $\ULCD$.
			
			For short, the method below works as follows: for $\Bw'$ (viewed as vectors in $\Q^n$) we have $M\Bw' \in \Z^{n}$, where $\Bw'=(w_1/p,\dots, w_n/p)$. Then we approximate this vector by an element from the obtained net, and then pass to consider the probability from each net element. After approximation, we have that $M\Bu'$ is close to $\Z^n$ in $\ell_2$-norm, and so we can apply the classical Erd\H{o}s-Tur\'an bound.
			
			Now we complete the proof of the proposition. Assume otherwise, then by the argument above, by passing to an appropriate $t\Bw$, we can assume that $\kappa/2 \le \|\Bw'\|_2 <\kappa$, and that $\Bw' \in S_D$ for some $D_i$ from $O(\kappa^2)$ dyadic intervals. As $\Bw$ is orthogonal to $X_1,\dots, X_{n-1}$ in $\F_p$, we then have the following key property for $\Bw'=\frac{1}{p} \Bw$  
			$$M\Bw' \in \Z^{n-d},$$
			where $M$ is the $n\times (n-d)$ matrix formed by $X_1,\dots, X_{n-1}$.
			
			By Lemma \ref{lemma:net:iid}, there exists $\Bu'\in \CN$ such that 
			$$\|\Bw'- \Bu'\|_2 =O(\kappa/D).$$
			It is well known that $\|M\| = O(\sqrt{n})$ with probability at least  $1-\exp(-\Theta(n))$ (We note that this is the only place where we used $M=O(\sqrt{n})$ to prevent the net from expanding), and so we will condition on this event. We then have
			\begin{equation}\label{M:norm}
			\|M^T \Bw' -M^T \Bu'\|_2 \le O(\sqrt{n} \kappa/D).
			\end{equation}
			Therefore,
			$$ \dist(M^T \Bu', \Z^{n-d})  \le O(\sqrt{n} \kappa/D).$$
			Let $\CE$ be this event, whose probability will be bounded shortly.
			By Theorem \ref{theorem:mod1}, as obviously $\kappa/D > 1/D$, we have
			\begin{align}\label{eqn:vectorpassing:iid} 
			\P(\| X_i \cdot \Bu' \|_{\R/\Z} = O(\kappa/D)) & = O(\kappa/D + (\log (D/\kappa)) (\exp(-\kappa^2) + \exp(-4\gamma^2 \|\Bu'\|_2^2)) \nonumber \\
			&=O(\kappa/D),
			\end{align}
			where in the last estimate we used the fact that $D \le \exp(c' \kappa^2)$ with sufficiently small $c'$.
			
			By Lemma \ref{lemma:tensor} we thus have for some absolute positive constant $C'$
			$$\P(\CE) \le (C' \kappa /D)^{n-d}.$$
			Putting together using union bound over all $\Bu'$ from the net, as $\kappa = n^c$, we obtain in the case $\kappa D \ge c_0 \sqrt{n}$ a bound 
			\begin{align*}
			\P(\exists \Bu' \in \CN, \|M \Bu'\|_{\R/\Z} =O(\sqrt{n} \kappa/D)) & \le D (C \kappa D/\sqrt{n} )^n \times (C' \kappa /D)^{n-d}\\  
			& \le D^2(CC'\kappa^2/\sqrt{n})^n < \exp(-\Theta(n)).
			\end{align*}
			Note that here we have to assume $\kappa =o(n^{1/4})$ at least.
			
			Also, in the second case that $\kappa D < c_0 \sqrt{n}$, noting that $D \ge k^{5/4-c}$
			\begin{align*}
			\P(\exists \Bu' \in \CN, \|M \Bu'\|_{\R/\Z} =O(\sqrt{n} \kappa/D))  &\le n^{c_0 n}  \times (C' \kappa /D)^{n-d} \\ 
			& \le n^{c_0 n}  \times (C'/\kappa)^{(1/4-c) n} \\
			&\le {C'}^n n^{c_0 n}  \times n^{-c(1/4-c) n} \\
			& \le n^{-cn/8},
			\end{align*}
			assuming that $c_0$ is sufficiently large compared to $c$, and that $c\le 1/16$.
		\end{proof}
		
		In our last result of this subsection, by using the terminology of Section \ref{section:combinatorial}, we show the following.  
		
		\begin{prop}[Normal vectors cannot  have combinatorial structure]\label{prop:structure:combinatorial} Assume that $p \le \exp(c \kappa^2)$ and that $\kappa =n^c$ for $c<1/16$. Let $X_1,\dots, X_{n-d}$ be the first $n-d$ columns of a matrix $M$ whose entries are iid copies of an $\a$-balanced random variable, where $d\le c'' n$ for some sufficiently small constant $c''$. Let $\Bw$ be any non-zero vector that is orthogonal to $X_1,\dots, X_{n-d}$. Then with probability at least $1 -\exp(-\bar{c} n)$, we have that 
			$$\rho(\Bw) \le e^{-\hat{c} \kappa^2}$$
			where $c'', \bar{c}$ and $\hat{c}$ are constants that only depend on $c$ and $\alpha$.
		\end{prop}
		
		Note that this result holds for $\al$-balanced ensembles where we don't have to assume $\|M\|_2 =O(\sqrt{n})$.

		Let $Z=(z_1,\dots,z_n)$ be any vector in $\F_p^n$. We first record the following elementary relation (where we recall $\rho(.)$ from Theorems \ref{theorem:ILO} and \ref{theorem:LO}).
		\begin{fact}\label{fact:comparison:rho} For any $I \subset [n]$ we have
			$$\rho(Z_I) \ge \rho(Z).$$
		\end{fact}
		\begin{proof} It suffices to show this for $I = [k]$. We first write
			\begin{align*}
			\P(z_1w_1+\dots +z_n w_n=r) -1/p  &= \E_{w_{k+1},\dots, w_n} [\P(\sum_{1\le i\le k}z_i w_i =r -\sum_{k+1\le i\le n}z_i w_i | w_{k+1},\dots,w_n )-1/p]\\
			&\le \max_{r'} \P(\sum_{1\le i\le k}z_i w_i =r')-1/p,
			\end{align*}
			and hence 
			$$\max_r \P(z_1w_1+\dots +z_n w_n=r) -1/p \le  \max_{r'} \P(\sum_{1\le i\le k}z_i w_i =r')-1/p,$$ 
			where we note that both sides are non-negative. 
			
			We can bound the minimum in a similar fashion 
			\begin{align*}
			\P(z_1w_1+\dots +z_n w_n=r) -1/p  &= \E_{w_{k+1},\dots, w_n} [\P(\sum_{1\le i\le k}z_i w_i =r -\sum_{k+1\le i\le n}z_i w_i) | w_{k+1},\dots,w_n )-1/p]\\
			&\le \min_{r'} \P(\sum_{1\le i\le i}z_i w_i =r')-1/p,
			\end{align*}
			and so
			$$\min_r \P(z_1w_1+\dots +z_n w_n=r) -1/p \le  \min_{r'} \P(\sum_{1\le i\le k}z_i w_i =r')-1/p,$$
			where we note that both sides are non-positive. 
			
			Putting this together, we thus obtain
			$$\max_r |\P(z_1w_1+\dots +z_n w_n=r) -1/p| \le  \max_{r'} |\P(\sum_{1\le i\le k}z_i w_i =r')-1/p|,$$ completing the proof.
		\end{proof}
		
		We next need the following key definitions and results from \cite{FJLS}.
		
		\begin{definition}
			For an $\Ba \in \F_p^n$, $k \in \mathbb{N}$ and $\delta \in [0,1]$, we define $R_k^\delta(\Ba)$ to be the number of solutions to 
			$$
			\pm a_{i_1} \pm a_{i_2} \cdots \pm a_{i_{2k}} = 0 \, \text{mod} \, p
			$$
			that satisfy $|\{i_1, \dots, i_{2k}\}| \geq (1+ \delta)k$.
		\end{definition}
		
		We will make use of the observation from \cite{FJLS} that $R_k(\Ba)$ is never much larger than $R_k^{\delta}(\Ba)$. 
		\begin{lemma}[Lemma 1.6, \cite{FJLS}]
			For all integers $k, n$ with $k \leq n/2$ and any prime $p$, $\Ba \in \F_p^n$ and $\delta \in (0,1)$,
			$$
			R_k(\Ba) \leq R_k^{\delta}(\Ba) + (40 k^{1-\delta} n^{1+\delta})^k.
			$$
		\end{lemma}
		
		As we will have the occassion to deal with subsets of vectors which we consider as vectors in their own right, we introduce the notation $|\Ba|$ to mean the dimension of a vector $\Ba$. By $\Bb\subset \Ba$ we mean that $\Bb$ is a truncation of $\Ba$. The key technical result in \cite{FJLS} is the following combinatorial lemma, which helps control the number of vectors with many ``local" arithmetic relations.
		\begin{lemma}\label{lemma:generalcounting} \cite[Theorem 1.7]{FJLS} Denote 
			$$\BB_{k,s,\ge t}^\delta:= \Big \{\Ba \in \F_p^n, R_k^\delta(\Bb) \ge t \frac{2^{2k} |\Bb|^{2k}}{p} \mbox{ for every $\Bb \subset \Ba$ with $|\Bb|\ge s$} \Big \}.$$
			Then 
			$$|\BB_{k,s,\ge t}^\delta| \le (\frac{s}{n})^{2k-1} (\delta t)^{s-n} p^n.$$
		\end{lemma} 
		At this point, we fix $\delta = 1/2$, $k = \lceil n^{1/8} \rceil$ and define 
		$$H_t=\Big \{\Ba \in \F_p^n, \exists \Bb \subset \Ba, \supp(\Bb) \ge n^{1/4}, R_k^{1/2}(\Bb) \le t \frac{2^{2k} |\Bb|^{2k}}{p}\Big \}.$$
		So roughly speaking, this is the set of $\Ba$ which are not arithmetically rich. As Lemma \ref{lemma:generalcounting} suggests, this set captures most of the vectors. More precisely  we have the following (see also \cite{FJLS}).
		
		\begin{corollary} \label{cor:counting} If $p \leq \exp(c \kappa^2)$ and $\kappa = n^c$ for $c < 1/16$, then for $t\ge n^{1/16}$ 
			$$| \{\Ba, |\supp(\Ba)| \ge n^{1/4}, \Ba \notin H_t\} | \le \left(\frac{4p}{t} \right)^n t^{n^{1/4}}.$$
		\end{corollary}
		\begin{proof}(of Corollary \ref{cor:counting})
			We can assume that $t \leq p$, otherwise the statement is trivially true as the left-hand side is zero.  Fix a subset $S \subset [n]$ with $|S| \geq n^{1/4}$ and enumerate the vectors $\Ba$ with $\supp(\Ba) = S$.  By assumption, $\Ba \notin H_t$ so the restriction $\Ba|_S$ of $\Ba$ to the set $S$ is an element of $\BB_{k,n^{1/4}, \geq t}(|S|)$.  Therefore, Lemma \ref{lemma:generalcounting} guarantees that the number of possible choices for $\Ba|_S$ is at most
			$$
			\left(\frac{n^{1/4}}{|S|} \right)^{2k-1} \left( \frac{2p}{t} \right)^{|S|} (t/2)^{n^{1/4}} \leq \left(\frac{2p}{t} \right)^n t^{n^{1/4}}
			$$
			where the second inequality follows from our assumption that $t \leq p$.  We obtain the final result by summing over all subsets $S$.
		\end{proof}

		The next lemma is a simple consequence of Theorem \ref{theorem:Halasz}. 
		\begin{lemma} \label{lemma:smallball} Suppose that $\Ba \in H_t$. If $p \leq \exp(c \kappa^2)$ and $\kappa = n^c$ for $c < 1/16$, then if $t\ge n^{1/16}$, there exists a constant $C > 0$ such that
			$$\rho(\Ba)\le \frac{Ct}{p n^{1/8}}.$$
		\end{lemma}
		
		\begin{proof} (of Lemma \ref{lemma:smallball})
			Let $\Bb$ be a subvector of $\Ba$ with $|\supp(\Bb)| \geq n^{1/4}$ and $R_k^{\delta}(\Bb) \leq t 2^{2k} |b|^{2k}/p$.  In the notation of Theorem \ref{theorem:Halasz}, if we let $f(x) = \sqrt{x}$ then
			\begin{align*}
			\rho(\Bb) &\leq \frac{R_k(\Bb)}{2^{2k} |\Bb|^{2k} n^{1/8}} + e^{-n^{1/8}/2} \\
			&\leq \frac{R_k^{1/2}(\Bb) + (40 k^{1/2} |\Bb|^{3/2})^k}{2^{2k} |\Bb|^{2k} n^{1/8}} + e^{-n^{1/8}/2} \\
			&\leq \frac{t 2^{2k} |\Bb|^{2k}/p + (40 k^{1/2} |\Bb|^{3/2})^k}{2^{2k} |\Bb|^{2k} n^{1/8}} + e^{-n^{1/8}/2} \\
			&\leq \frac{t}{p n^{1/8}} + \left(\frac{20 k}{|\Bb|} \right)^{k/2}  + e^{-n^{1/8}/2}. 
			\end{align*}   
			This expression is dominated by the first term by our bound on the range of $p$ and our choice of $k$.  We recall that
			$$
			\rho(\Ba) \leq \rho(\Bb)
			$$ to finish the proof. 
		\end{proof}
		Now we complete one of our main results of the section.
		\begin{proof}(of Proposition \ref{prop:structure:combinatorial}) We let $\CV$ denote the vectors in $\F_p^n$ with support larger than $c_{nsp} n$ and $\CW$ the set of non-zero vectors $\Bw \in F_p^n$ that such that 
			$$\rho(\Bw) \geq e^{-\hat{c} \kappa^2}.$$  
			By Corollary \ref{cor:smallp}, we can assume $p \gtrsim \sqrt{n}$.
			Observe that
			$$
			\P(\exists \Bv \in \CW, \Bv \perp X_1,\dots, X_{n-d}) = \P(\exists \Bv \in \CW \cap \CV, \Bv \perp X_1,\dots, X_{n-d} ) + \P(\exists \Bv \in \CW \cap \CV^c, \Bv \perp X_1,\dots, X_{n-d}).
			$$
			
			By Lemma \ref{nonsparselem},
			\begin{align*}
			\P(\exists \Bv \in \CW, \Bv \in \CV^c) \le \exp(- \Theta(n)).
			\end{align*}
			
			Therefore, it suffices to focus on the vectors in $\CV$.  Note that any $\Bw \in \CV$ must reside in $H_p$ since $R_k^{1/2} \leq 2^{2k} |\Bb|^{2k}$.  
			
			There are two cases to consider.
			
			\noindent
			{\bf Case 1.}  We begin with vectors in $\CV \cap \CW \cap H_{n^{1/16}}$. Let $\Ba$ be such a vector, then by definition of $\CW$ and By Lemma \ref{lemma:smallball}, 
			$$
			 e^{-\hat{c} \kappa^2} \le \rho(\Ba) \leq \frac{1}{p}.
			$$
			Because of the lower bound, by Theorem \ref{theorem:degenerate} there exists a generalized arithmetic progression $P$ of rank one in $\F_p$ and of size $O(p/n^\eps)$ that contains all but $n^{2\eps}$ entries of $\Ba$. Note that the number of ways to choose such a $P$ is bounded by $p^{O(1)}$.  For a fixed $P$, the number of vectors $\Ba$ with at least $n - n^{2 \eps}$ components in $P$ is at most
			$$
			\binom{n}{n^{2 \eps}} |P|^{n - n^{2 \eps}} p^{n^{2\eps}} \leq 2^n (p/n^{\eps})^{n- n^{2 \eps}} p^{n^{2 \eps}} \leq 8^n (p/n^{\eps})^{n}.
			$$
			The probability that any 	$\Ba \in \CV \cap \CW \cap H_{n^{1/16}}$ with at least $n- n^{2 \eps}$ components in $P$ is orthogonal to $X_1, \dots, X_{n-d}$ is bounded by
			$$
			8^n (p/n^{\eps})^n (\frac{1}{p} + \frac{C n^{1/16}}{p n^{1/8}} )^{n-d} \leq 8^n (p/n^{\eps})^n (2/p)^{n-d} \leq n^{-\eps n/2}, 
			$$
			provided that $d \leq cn$ for some small constant $c$.  Finally, we take a union bound over $p^{O(1)}$ choices of $P$ to conclude the proof. 

			\noindent 
			{\bf Case 2.} We address the remaining vectors.  Let $\tau = n^{1/16}$.  We show that no vector in $(\CW \cap \CV) \setminus H_{n^{1/16}}$ is orthogonal to $X_1, \dots, X_{n-d}$.  We can now partition  $(\CW \cap \CV) \setminus H_{n^{1/16}}$ as
			$$
			\bigsqcup_{j=1}^{J} H_{2^j \tau} \setminus H_{2^{j-1} \tau}
			$$
			where $J$ is the smallest integer such that $2^{J} \tau \geq p$.  Clearly, $J \leq \kappa^2$.
			We then have
			\begin{align*}
			\P(\exists v \in (\CW \cap \CV) \setminus H_{n^{1/16}}, \Bv \perp X_1,\dots, X_{n-d} ) &= \sum_{j=1}^J \P(\exists \Bv \in (\CW \cap \CV) \cap(H_{2^j \tau} \setminus H_{2^{j-1} \tau}), \Bv \perp X_1,\dots, X_{n-d}). 
			\end{align*} 
			Combining Corollary \ref{cor:counting} and Lemma \ref{lemma:smallball}, we have (noting trivially that $2^j \tau \ge n^{1/16}$)
			\begin{align*}
			\sum_{j=1}^J \P(\exists \Bv \in \CW \cap (H_{2^j \tau} \setminus H_{2^{j-1} \tau}), \Bv \perp X_1,\dots, X_{n-d} ) &\leq \sum_{j=1}^J \left( \frac{4 p}{2^{j-1} \tau} \right)^n (2^{j-1} \tau)^{n^{1/4}} \left(\frac{1}{p}+\frac{C 2^j \tau}{p n^{1/16}} \right)^{n-d}\\
			&\leq \sum_{j=1}^J \left( \frac{4 p}{2^{j-1} \tau} \right)^n (2^{j-1} \tau)^{n^{1/4}} \left(\frac{2 C 2^j \tau}{p n^{1/16}} \right)^{n-d}\\
			&\leq n^{-(n-d)/16} C^n p^d \sum_{k=n+1}^p (2^{j-1} \tau)^{n^{1/4}} \\
			&\leq \kappa^2 n^{-n/32}  p^{d+ 2 n^{1/4}} \\
			&\leq \exp(- n \log n/ 64)
			\end{align*} 
			where the last line follows from small enough $c''$.

			Combining the above estimates, we can conclude that
			$$
			\P(\exists \Bv \in \CW , \Bv \perp X_1,\dots, X_{n-d}) \leq \exp(-\Theta(n)) + \exp(-c' n \log n),
			$$ as desired. \end{proof}


		\section{Distribution of ranks revisited}\label{section:rank} In this section we give a short proof for Theorem \ref{thm:Maple'}. We start with a high-dimensional lemma, which, in some sense, is a discrete analog of \cite{RV-rec} where they considered distance of a random vector to a subspace of condimension $d$ in $\R^n$.

		
		
		

		\begin{lemma}\label{lemma:highdim} Assume that $H$ is a subspace in $\F_p^n$ of codimension $d$, and such that for any $\Bw \in H$ we have $\rho(\Bw) \le \delta$. Then 
			$$|\P(X \in H) - 1/p^d|\le \delta.$$
		\end{lemma}
		
		\begin{proof}(of Lemma \ref{lemma:highdim}) Let $\Bv_1,\dots, \Bv_d$ be a basis of $H$. Our assumption says that for any $t_1,\dots, t_d$, not all zero, we have 
			$$\rho(\sum_i t_i \Bv_i ) \le \delta.$$
			Note that in $\F_p^n$
			$$1_{a_1=0,\dots, a_d=0} = p^{-d} \sum_{t_1,\dots, t_d \in \F_p}  e_p( t_1 a_1 +\dots+ t_d a_d).$$
			We have 
			\begin{align*}
			\P(\wedge_{i=1}^d X \cdot \Bv_i = 0) &= p^{-d} \sum_{t_1,\dots, t_d \in \F_p} \frac{1}{p}\sum_{t\in \F_p} e_p(t[t_1 X \cdot \Bv_1 +\dots+ t_d X  \cdot \Bv_d)])\\
			&=p^{-d} + p^{-d}\sum_{t_i \in \F_p, \mbox{ not all zero}} \frac{1}{p}  \sum_{t \in \F_p}  e_p(t (X \cdot \sum_i t_i \Bv_i)).
			\end{align*}
			Now by our assumption 
			$$ | \frac{1}{p}  \sum_{t \in \F_p}  e_p(t (X \cdot \sum_i t_i \Bv_i)) | \le \delta.$$
			Thus we have
			\begin{align*}
			|\P(\wedge_i X \cdot \Bv_i = 0)- p^{-d}| &\le \delta,
			\end{align*}
			completing the proof. \end{proof}
		
		Now we apply Propositions \ref{prop:GAPstructure:iid}, \ref{prop:structure:iid} and \ref{prop:structure:combinatorial} to prove the following. 
		
		\begin{lemma}\label{lemma:rank:iid:d} Assume that $p \le \exp(c\kappa^2)$ and that $\kappa =n^c$ for $c<1/16$ and $0\le d, u\le n^c$. There exists an event $\CE_d$ with probability $\P(\CE) \ge 1 -e^{n^{c'}}$ such that the following holds 
			$$\Big|\P\big(X \in W_{n-u}|\CE \wedge \rank(W_{n-u})=n-u-d\big)-1/p^{u+d}\Big| \le \exp(-n^{c'}),$$
			where $W_{n-u}$ is the subspace generated by $X_1,\dots, X_{n-u}$.
		\end{lemma}
		Assume this Lemma, we can then complete Theorem \ref{thm:Maple'} by direct calculations, or by applying  \cite[Theorem 5.3]{NgW1}, we leave it for the reader as an exercise.  
		
		\begin{proof}(of Lemma \ref{lemma:rank:iid:d}) We have seen from Propositions \ref{prop:structure:iid} and \ref{prop:structure:combinatorial} that there is an event $\CE$ with $\P(\CE) \ge 1 - e^{-n^{c'}}$ such that for any $\Bw \in W_{n-u}$ we have 
			$$\rho(\Bw) \le e^{-n^{c'}}.$$ 
			Now conditioning on this event, if $\rank(W_{n-u})=n-u-d$ then the codimension of $W_{n-u}$ is $u+d$, and hence by Lemma \ref{lemma:highdim} we have 
			$$|\P(X \in W_{n-u})- \frac{1}{p^{d+u}}| \le e^{-n^{c'}},$$ as claimed.
		\end{proof}

		\section{Equi-distribution of the normal vectors}\label{section:normal:equi} In this section we prove Theorem \ref{thm:entrywise}. For convenience we decompose the task into two parts.

		\begin{proposition}\label{prop:entrywise} With the same assumption as in Theorem \ref{thm:entrywise} we have
			\begin{itemize}
				\item For each $i \in \{1, \cdots, n\}$, we have 
				$$|\P(w_i = 0) - 1/p| \leq O(\exp(-n^{c'})).$$
				\item For each $i \neq j$, and for any $a\in \F_p$ we have
				$$|\P(\exists \Bw=(w_1,\dots, w_n) \in W_{n-1}^{\perp}: w_i = a \wedge w_j=1) - 1/p| \leq O(\exp(-n^{c'})).$$
			\end{itemize}
		\end{proposition}

		\begin{proof}(of Proposition \ref{prop:entrywise}) 
			We prove the first item of Proposition \ref{prop:entrywise}. It suffices to show this for $i=1$. Fix $i = 1$. We seek to bound the probability of the event that our normal vector $\Bv$ has $v_1 = 0$ under the condition that our first $n - 1$ columns achieve full rank, i.e. $\rank{(M_{n \times (n - 1)})} = n-1.$ Suppose we are given such a normal vector. Then restricting to the bottom $n - 1$ rows, we see that this is equivalent to the event that the submatrix $M_{(n - 1) \times (n - 1)}$ has
			 a nontrivial nullspace. So we rewrite $\P(v_1 = 0 \, | \, \rank{(M_{n \times (n - 1)})} = n-1)$ as $\P(M_{(n - 1) \times (n - 1)}$ is singular $| \, \rank{(M_{n \times (n - 1)})} = n-1)$. We can simply view this as $\P(\rank{(M_{(n - 1) \times (n - 1)}}) = n - 2$ $| \, \rank{(M_{n \times (n - 1)})} = n - 1) := \P(A | B) = \P(A \cap B)/P(B).$ 
			
			By Theorem \ref{thm:Maple'} (see also \cite[Theorem A.4]{NgP}), we know that 
			$$\P(B) = \prod_{i=2}^\infty (1-p^{-i}) +O(e^{-n^c}).$$
			
			Now consider the event $A \cap B.$ This is the event that rows $\Br_2, \cdots, \Br_n$ span a subspace $H$ of dimension $n - 2$ and $\Br_1$ is not in the span of $H$, which can be expressed as $\P(A' \cap B' )= \P(A')\P(B' | A')$. For $\P(A')$, we again use Theorem \ref{thm:Maple'},
			$$\P(A') = \frac{1}{p} \frac{\prod_{i=2}^\infty (1-p^{-i})}{(1-p^{-1})} +O(e^{-cn}).$$
			
			For $\P(B' | A')$, our previous section tells us that if we condition on rows $\Br_2, \cdots, \Br_{n}$ having rank equal to $n-2$, then our normal vector $\Bv'$ exists and has large $\ULCD$, i.e. $\rho(\Bv') \leq \exp(-c\kappa^2)$. So the probability that $\Br_1$ is in the span of $H$ under this condition is 
			$$\P((B')^c|A') = 1/p+O(\exp{(-n^{c'})}).$$
			
			Putting this all together, we have \begin{align*}
			\P(A | B) &= \frac{\P(A \cap B)}{\P(B)} \\
			&= \frac{\P(A')\P(B' | A')}{\P(B)} \\
			&= \frac{1}{p} +  O(e^{-n^{c'}}). 
			\end{align*}

			Now we prove the second item. It suffices to assume $(i,j) = (1,2)$. The event that $w_1 = a,w_2=1$ is equivalent to the event that $a\Br_1 + \Br_2 + \Br_3w_3 \cdots + \Br_nw_n = 0$, where $\Br_i$ is the $i$-th row of our matrix. If $a$ is zero, we are done via the previous argument, so assume $a$ is nonzero. 
			
			Let $H$ be the span of rows $\Br_3, \cdots, \Br_n$, which has full rank $n - 2$ by our rank assumption on $M_{n \times (n-1)}$ and the fact that $a\Br_1 + \Br_2 + \Br_3w_3 \cdots + \Br_nw_n = 0.$. Further, let $\pi$ be the projection to the orthogonal complement $H^\perp.$
			For each evaluation of rows $\Br_3, \cdots, \Br_n$, $\pi$ is deterministic and $\pi(\Br) = \langle \Br, \Bn \rangle,$ where $\Bn$ is the deterministic normal vector. Applying this projection to the linear combination, we have $$a \langle \Br_1, \Bn \rangle + \langle \Br_2, \Bn \rangle = 0.$$ Since $\Bn$ is deterministic, each inner product takes values $b,c \in \BF_p$ with probability uniformly $1/p$ with error $O(\exp{(-n^c)})$ by Theorem \ref{prop:structure:iid}. This means that $a$, as the ratio of the inner products, is also uniformly distributed with probability $1/p$ and similar error. \end{proof}
		
		We next prove the second part of Theorem \ref{thm:entrywise}, restated for convenience. 
		
		\begin{prop}\label{prop:counts} Assume that $p\lesssim n/\log n$. Let $n_a$ denotes the number of $w_i$ such that $w_i=a$. Then for any $\delta<1$ which might depend on $n$ such that $\delta^{-2}p=o(n/\log n)$. We then have
			$$\P(\wedge_{a=0}^{p-1}(|n_a/n - 1/p| \le \delta/p) \ge 1 -e^{-c\delta^2 n/p}.$$
		\end{prop}
		
		\begin{proof}(of Proposition \ref{prop:counts}) Let $\CE$ denote the set of vectors under consideration up to scaling, and $\bar{\CE}$ be the complement. We first have the following elementary fact.
			
			\begin{fact}\label{NETheorem} We have
				$$|\bar{\CE}| \leq 2p^{n} e^{-c\delta^2 n/ p}.$$
			\end{fact}
			
			\begin{proof}
				First, we note that each $n_a$ has distribution Bin($n, \frac{1}{p}$) with variance $$n(\frac{1}{p})(1 - \frac{1}{p}) = \frac{n(p - 1)}{p^2}.$$ Letting $0 < \delta < 1$ and $\mu$ denote the mean of this distribution, the upper-tail and lower-tail Chernoff inequalities combine to form the following bound: $$\P(|n_a - \mu| \geq \delta \mu) \leq 2e^{-c \mu \delta^2}.$$ 
				
				
				For each $i$ in $\{0, 1, \cdots, p-1\}$, let $F_a$ denote the event that $|n_a- \frac{n}{p}| < \delta n /p.$ Then trivial union bound gives $$\P(F_0 \cap \cdots \cap F_{p-1}) \geq 1 - 2pe^{-c \delta^2 n/p}.$$ 
				Since there are $p^{n}$ different choices for $\Bv$, the number of non-equidistributed vectors $\Bv$ is at most $p^{n} (2p e^{-c\delta^2 n/p})$.
			\end{proof}
			
			Now we complete our result. Let $\Bw$ be an arbitrary vector in $\BF_p^n$. We seek to upper bound $\P(\Bw$ is normal and $\Bw \in \bar{\CE}$). Immediately we have $\P(\Bw$ is normal and $\Bw \in \bar{\CE})$ is bounded above by $$\sum_{\Bw \in \bar{\CE}} \P(\Bw \perp X_1, \cdots, X_{n-1}) \leq \sum_{\Bw \in \bar{\CE}} (\rho(\Bw))^{n-1}.$$ 
			
			Similar to our previous sections, we may decompose the sum into classes where $\Bw$ is sparse and $\Bw$ is non-sparse. By Lemma \ref{nonsparselem}, the contribution over our sparse vectors is negligible. For our non-sparse vectors, we appeal to Theorem \ref{prop:structure:iid}. We can now bound the sum via: 
			\begin{align*}
			\sum_{\Bw \in (\bar{\CE}), non-sparse}(\rho(\Bw))^{n-1} &\leq \frac{2p^{n}}{e^{\delta^2 n/ 3p}}(1/p + e^{-n^c})^{n-1} \\
			&= \frac{2p}{e^{\delta^2n/3p}}(1 + pe^{-n^c})^{n-1} \\
			&\leq \frac{4p}{e^{c\delta^2 n/p}} \\
			&= O(e^{-c\delta^2n/2p}),
			\end{align*}
			as  long as $\delta^{-2}p=o(n/\log n) $, completing the proof.
		\end{proof}

		\section{Proof of Theorem \ref{thm:char:universal}}\label{section:polynomials}
		
		It suffices to prove Theorem \ref{thm:char:universal'} below. Again, for simplicity we will assume $M_n$ to be an iid Bernoulli matrix taking values $\pm 1$ with probability 1/2 and $p\ge 3$. 
		We recall that $\phi(x)$ has degree $d$ and $p$ is sufficiently large and
		$$p \le n^{1/2 -\eps}.$$
		Let $\al$ be a root of $\phi(\al) =0$ and consider the field extension $\F_q=\F_p[\al]$. Notice that any element $x$ of this field has form 
		$$x=\sum_{i=0}^{d-1} c_i \al^i.$$ 
		More importantly, the event $\phi(x) | D_M(x)$ is equivalent to the event that $M_n -\a$ has rank at most $n-1$ in this field $\F_q$. In other words, let $W_{n-k}$ be the subspace in $\F_q^n$ generated by the first $n-k$ columns of the matrix $M_n -\al$ (equivalently, the columns of $M_{[n] \times [n-k]}-\al$), then the event that $M_n-\a$ has rank at most $n-1$ is the union of the (disjoint) events $\CE_k$  that $W_{n-k}$ has rank $n-k$ and the $n-k+1$-th column $X_{n-k+1} - \al \Be_{n-k+1}$ belongs to $W_{n-k}$. In what follows we will be mainly focusing on the case of $\CE_1$, treatments for $\CE_2, \CE_3, \dots$ will be discussed later, and summing over these events will imply Theorem \ref{thm:char:universal}.
		
		\begin{theorem}\label{thm:char:universal'} There exists an absolute constant $C$ such that
			$$|\P(\CE_1)- \frac{1}{p^d} | = |\P( X_n -\a \Be_n \in W_{n-1} | \rank(W_{n-1})=n-1) - \frac{1}{p^d}|  \le C e^{-n^{1-\eps}/p^2}.$$
		\end{theorem}

		Consider the normal vector $\Bv=(v_1,\dots, v_n)$ of $W_{n-1}$ (i.e. the column space of the matrix $M_{[n] \times [n-1]} -\al$). This vector can be written as 
		$$\Bv = \sum_{i=0}^{d-1} \al^i \Bu_i,$$
		where $\Bu_i = (u_{1i},\dots, u_{ni})^T \in \F_p^n$.
		Notice that as for $1\le j\le n-1$ we have $\Bv \cdot  (X_j - \a \Be_j)=0$. So we have that 
		$$\sum_{i=0}^{d-1} \al^i (\Bu_i \cdot X_j) = \a^{i+1} \Bu_i \cdot  \Be_j.$$
		This implies that
		\begin{equation}\label{eqn:u:cir}
		\Bu_i \cdot X_j = u_{j(i-1)}.
		\end{equation}
		Let 
		\begin{equation}\label{eqn:Q}
		Q= M_{[n-1] \times [n-1]}^T.
		\end{equation}
		By fixing the last coordinates $u_{in}=f_i \in \F_p$ of each $\Bu_i$ (and with a loss of a multiplicative factor $p^d$ in probability), and by fixing $X = \row_n(M_{[n] \times [n-1]})$ (i.e. $X$ is the last row of $M_{[n] \times [n-1]}$), with $\Bv_i$ be the truncated vectors $(u_{1i},\dots, u_{(n-1)i})$ we can rewrite \eqref{eqn:u:cir} as (with $\Bv_d=\Bv_0$) 
		\begin{equation}\label{eqn:Q:vi}
		Q\Bv_i + f_i X = \Bv_{i-1}, 1\le i \le d.
		\end{equation}
		In what follows we set
		$$m:=n^{1-\eps/2}.$$

		Conditioning on $f_i, X$, we will show the following key lemma. 
		\begin{lemma}\label{lemma:nonsp:char} With probability $\exp(-\Theta(n))$ with respect to the columns of the matrix $M_n$, for any vector $\Bv$ that is orthogonal to the first $n-1$ columns of $M_n -\al$ the subspace $H$ generated by $\Bv_1,\dots, \Bv_d$ (defined as above) in $\F_p^n$ cannot have $m$-sparse vector. In other words, there do not exist coefficients $\a_i \in \F_p$, not all zero, such that 
			$$|\supp(\sum_i \a_i \Bv_i)| \le m.$$ 
		\end{lemma}
		We can actually prove a slightly more general version of this lemma,  which will be used to control $\CE_2,\CE_3,\dots$.

		\begin{lemma}\label{lemma:nonsp:char'} With probability at most $\exp(-\Theta(n))$ with respect to the columns of the matrix $M_n$, for any vector $\Bv$ that is orthogonal to the first $n'$ columns of $M_n -\al$ (where $n'=(1-o(1))n$) there do not exist coefficients $\a_i \in \F_p$, not all zero, such that 
			$$|\supp(\sum_i \a_i \Bv_i)| \le m.$$  
		\end{lemma}
		We remark that the above lemmas are somewhat similar to Propositions \ref{prop:GAPstructure:iid}, \ref{prop:structure:iid}, \ref{prop:structure:combinatorial}, but the situation here is much more complicated as the relation between $\Bv_i$ and $X_1,\dots, X_{n'}$ are non-trivial (for instance $\Bv_i$ is not orthogonal to $X_j$), and also the diagonal entries are perturbed by $\al$.  
		
		We postpone the proof of this lemma for a moment, and let us use it to prove the following result, which automatically implies Theorem \ref{thm:char:universal'}.  
		
		\begin{theorem}\label{thm:char'} On the event of Lemma \ref{lemma:nonsp:char}  we have
			$$|\P_{X_n}((X_n  - \al \Be_n) \cdot \Bv = 0)-1/p^d| \le \exp(-c m/p^2).$$
		\end{theorem}
		
		\begin{proof}(of Theorem \ref{thm:char'}) Notice that the event $(X_n -\al \Be_n) \cdot \Bv =0$ implies that (by \eqref{eqn:u:cir}, using the same notations for $\Bv_i, f_i, x_n$) for all $1\le i\le d$
			$$\Bv_i \cdot X_n|_{[n-1]} + f_i x_n = f_{i-1}.$$  
			In other words, conditioning on $x_n$, and by letting $Y_n=  X_n|_{[n-1]}$ and by choosing deterministic numbers $g_i\in \F_p$ appropriately we have
			\begin{equation}\label{eqn:Yg}
			Y_n  \cdot \Bv_i =g_i, 1\le i\le d.
			\end{equation}
			\begin{align*}
			\P(\wedge_i Y_n \cdot \Bv_i = g_i ) &= p^{-d} \sum_{t_i \in \F_p} e_p(t_1(Y_n \cdot \Bv_1-g_1) +\dots+ t_d(Y_n  \cdot \Bv_d-g_d))\\
			& = p^{-d} \sum_{t_i \in \F_p} \frac{1}{p}\sum_{t\in \F_p} e_p(t[t_1(Y_n \cdot \Bv_1-g_1) +\dots+ t_d(Y_n  \cdot \Bv_d-g_d)])\\
			& = p^{-d} + p^{-d} \sum_{t_i \in \F_p,  \mbox{ not all zero}}e_p(t_1(Y_n \cdot \Bv_1-g_1) +\dots+ t_d(Y_n  \cdot \Bv_d-g_d))\\
			&=p^{-d} + p^{-d}\sum_{t_i \in \F_p, \mbox{ not all zero}} \frac{1}{p-1}  \sum_{t \in \F_p, t\neq 0}  e_p(t[Y_n \cdot \sum_i t_i \Bv_i  - \sum_i t_ig_i ]),
			\end{align*}		
			Now as $\sum_i t_i \Bv_i$ is not very sparse for any non-trivial choice of $(t_1,\dots,t_d)$, by Corollary \ref{cor:smallp} we have 
			$$ | \frac{1}{p}  \sum_{t \in \F_p, t\neq 0}   e_p(t[Y_n \cdot \sum_i t_i \Bv_i  - \sum_i t_ig_i ]) | \le \exp(-c m/p^2).$$
			Thus we have
			\begin{align*}
			|\P(\wedge_i Y_n \cdot \Bv_i = g_i )- p^{-d}| &\le \frac{p}{p-1} \exp(-c m/p^2).
			\end{align*}
		\end{proof}
		Notice that in the above proof, with $Z_n = X_n -\al \Be_n$, then the event $Z_n \cdot \Bv=0$ can be written as 
		$$\P(Z_n \cdot \Bv =0) = \frac{1}{q} \sum_{t \in \F_q} e_p(\tr(t Z_n \cdot \Bv)) = \frac{1}{q} + \frac{1}{q} \sum_{t \in \F_q, t \neq 0} e_p(\tr(t Z_n \cdot \Bv)),$$
		where $\tr : \F_q\to \F_p$ is the field trace. We just showed that 
		\begin{equation}\label{eqn:Z_n}
		|\P(Z_n \cdot \Bv =0) - \frac{1}{q}| = |\frac{1}{q} \sum_{t \in \F_q, t \neq 0} e_p(\tr(t Z_n \cdot \Bv))| \le  \exp(-c m/p^2).
		\end{equation}
		
		In the same way, we show the following more general version of Theorem \ref{thm:char:universal'}.
		
		\begin{theorem}\label{thm:char''} On the event of Lemma \ref{lemma:nonsp:char'}, as long as $k=o(n)$ we have
			$$|\P_{X_{n-k+1}}(X_{n-k+1}  - \al \Be_{n-k+1} \in W_{n-k}| \rank(W_{n-k})=n-k)-1/p^{dk}| \le \exp(-c m/p^2).$$
			In particularly,
			$$|\P(\CE_k) - \frac{1}{p^{dk}}|  \le \exp(-c m/p^2).$$
		\end{theorem}
		
		\begin{proof}(of Theorem \ref{thm:char''}) Notice that the event $Z_{n-k+1}= X_{n-k+1} -\al \Be_{n-k+1} \in  W_{n-k}$ is equivalent with the event that this vector is orthogonal to the (orthogonal basis) $\Bw_1,\dots, \Bw_{k}$ of $W_{n-k}^\perp$ in $\F_q^n$. We have
			\begin{align}\label{eqn:Z:q}
			\P(\wedge_{i=1}^{k} Z_{n-k+1} \cdot \Bw_{i} = 0) &= q^{-k} \sum_{t_{i} \in \F_q} e_p(\tr(t_{1}(Z_{n-k+1} \cdot \Bw_1)) +\dots+ \tr(t_{k}(Y_{n-k+1}  \cdot \Bw_{k}))) \nonumber \\
			&=q^{-k} + q^{-k}\sum_{t_i \in \F_q, \mbox{ not all zero}}  e_p(\tr(Z_{n-k+1} \cdot \sum_i t_{i} \Bw_i))  \nonumber \\
			&=q^{-k} + q^{-k}\sum_{t_i \in \F_q, \mbox{ not all zero}} \frac{1}{q-1}  \sum_{t \in \F_q, t\neq 0}  e_p(\tr( t Z_{n-k+1} \cdot \sum_i t_i \Bw_i)).
			\end{align}
			Now observe that $\Bv=\sum_i t_i \Bw_i$ is a non-zero vector that is orthogonal to the first $n'$ columns of $M_n -\al$. By Lemma \ref{lemma:nonsp:char'} and by the proof of Theorem \ref{thm:char'} (via Equation \eqref{eqn:Z_n}) we have 
			$$ |\frac{1}{q}  \sum_{t \in \F_q, t\neq 0}  e_p(\tr( t Z_{n-k+1} \cdot \sum_i t_i \Bw_i) | = |\frac{1}{q}  \sum_{t \in \F_q, t\neq 0}  e_p(\tr( t Z_{n-k+1} \cdot \Bv)| \le \exp(-c m/p^2).$$
			Plugging this bound into Equation \eqref{eqn:Z:q} for each non-zero tuple $(t_1,\dots, t_d)$ we complete the proof.

		\end{proof}
		For the rest of this section we will be focusing on  Lemma \ref{lemma:nonsp:char}. The proof of Lemma \ref{lemma:nonsp:char'} can be done similarly. Indeed, assume that $\Bw_i = \sum_{k=0}^{d-1} \al^k \Bv_{ik}$ with $\Bv_{ik} \in \F_p^n$, then $\al_i \Bw_i$ can be expressed as $\sum_{k=0}^{d-1} \al^k \al_{i k}  \Bv_{ik}$ for some $\al_{ik}\ \in \F_p$, one of which is non-zero if $\al_i$ is non-zero. Hence if $\sum_{i} \al_i \Bw_i = \sum_{k=0}^{d-1}\al^k (\sum_i \al_{ik}  \Bv_{ik})$ is $m$-sparse in $\F_q$, then $\sum_i \al_{ik}  \Bv_{ik}$ are $m$-sparse for any $0\le k\le d-1$. Choose one index $k$ where $\al_k$ is non-zero, and hence not all $\al_{ik}$ are zero.

		In what follows we prove the key lemmas on non-sparsity. We will mainly focus on Lemma \ref{lemma:nonsp:char} because the proof for Lemma \ref{lemma:nonsp:char'} is almost identical as long as $n'=(1-o(1))n$.
		
		\subsection{Proof of Lemma \ref{lemma:nonsp:char}}
		Let us assume that $\sum_i c_i \Bv_i$ is an $m$-sparse vector. We first note that by a proper ``rotation", we can assume that this vector is $\Bv_1$. This can be seen by, where $\Bv$ and $\Bu_i$ are as before, 
		$$(\sum_{i=0}^{d-1} c_i \a^i) \Bv = (\sum_{i=0}^{d-1} c_i \a^i) \sum_{i=0}^{d-1} \al^i \Bv_i = \sum_{i=0}^{d-1} \a^i \sum_{j=0}^{d-1} c_j \Bu_{i-j}.$$
		
		By iterating \eqref{eqn:Q:vi} we have (with $\Bg=X$ and $Q$ from \eqref{eqn:Q}, and $t_i$ are deterministic, being determined by $f_i$ from  \eqref{eqn:Q:vi})
		\begin{equation}\label{eqn:Q}
		Q^{d}\Bv_1 + \sum_{j=0}^{d-1} t_j Q^j \Bg  = \Bv_1.
		\end{equation}
		Our goal is that, assuming that $\Bv_1$ is $m$-sparse, then conditioning on a realization of $\binom{n}{m} p^m = \exp(o(n))$ possible values of $\Bv_1$, we show that the probability $\P(Q^{d}\Bv_1 + \sum_{j=0}^{d-1} t_j Q^j \Bg  = \Bv_1)$ is very small, so that after taking union bound the probability is still negligible (where we will use the assumption that $p \le n^{1/2-\eps}$ and $m \le n^{1-2\eps}$). 
		
		We will estimate $\P(Q^{d}\Bv_1 + \sum_{j=0}^{d-1} t_j Q^j \Bg  = \Bv_1)$ by a decoupling process, which roughly speaking allows us to pass from polynomials of $Q$ to multilinear forms where the factors are sparser matrices than $Q$, but they are independent, and so that we can control the probability easier. This process, roughly speaking, can be described as follows: assume that $M_i$ is a matrix obtained from $Q$ by replacing all entries by zero, except the $i$-th block of rows (in general we decompose the rows of $Q$ into $2^d$ groups, each with consecutive indices of size approximately $(n-1)/2^d$, and our matrices are formed by rows within a group). Then we can write 
		$$M_i = M_{i1}+ M_{i2},$$
		where $M_{i1}$ and $M_{i2}$ are submatrices obtained from $M_i$ by dividing the $i$-th block of rows into two groups of (almost) equal size.
		
		In general, in each step of our process, we decrease the polynomial degrees of a given matrix (the total degree remains the same), but double the number of matrices, and hence the probability. We will rely on the following well-known decoupling result (see for instance \cite{CTV,Ver}).
		
		\begin{claim}\label{claim:coupling} Let $X,Y$ be two random vectors in $\R^k$ and $\R^l$, and let $f: \R^{k+l} \to \R$ be a function. Then for any $a$ we have 
			$$\P_{X,Y}(f(X,Y)=a)^4 \le (\P_{X,Y,Y'}(f(X,Y)-f(X,Y')=0))^2$$ 
			$$\le  \P_{X,X',Y,Y'}(f(X,Y)-f(X,Y')-f(X',Y)-f(X',Y')=0),$$
			where $X'$ is an independent copy of $X$, and $Y'$ is an independent copy of $Y$. 
		\end{claim}
		
		In our application below $X,Y,X',Y'$, etc will be matrices. Now we describe the process in more details. 
		
		\begin{itemize}
			\item We start by decomposing $Q=M_1+M_2$ in \eqref{eqn:Q}. Factoring out, we will obtain a sum of many products of $M_1$ and $M_2$. In the next step we will be using Claim \ref{claim:coupling} to remove $M_1^d$ and $M_2^d$ (i.e. the highest degree polynomial of $M_1$ and $M_2$) accordingly. 
			\vskip .1in
			\item In general, assume that we have $XZ_1XZ_2\dots Z_kX$, where $Z_i$ are products of matrices that do not contain $X$ (it is possible that $Z_i$ is just the identity matrix), and $X$ appears $k$ times; we then decompose $X=X'+X''$ into block matrices, expanding the products we obtain
			$$(X'+X'')Z_1(X'+X'')Z_2\dots Z_k(X'+X'')=X'Z_1X'Z_2\dots Z_kX' + X''Z_1X''Z_2\dots Z_kX''+R,$$
			where in $R$ the total number of appearances of $X'$ and $X''$ are at most $k-1$. We then keep $X'$, decouple $X''$ by another independent matrix $Y''$ of the same distribution to remove $X'Z_1X'Z_2\dots Z_kX'$, and then keep $X''$ and $Y''$, decouple $X'$ by another independent matrix $Y'$ of the same distribution to remove $X''Z_1X''Z_2\dots Z_kX'' - Y''Z_1X''Z_2\dots Z_kY''$. More precisely, by Claim \ref{claim:coupling} (with $\{\Bz_1, \Bz_2\}= \{\Bf, \Bv_1\}$)
			
			\begin{align*}
			& \P_{X',X'',\dots}\Big (A(X',X'') \Bz_1+ (X'+X'')Z_1(X'+X'')Z_2\dots Z_k(X'+X'') \Bz_2 =0\Big)^4\\ 
			= & \P_{X',X'',\dots}\Big( A(X',X'') \Bz_1 + (X'Z_1X'Z_2\dots Z_kX' + X''Z_1X''Z_2\dots Z_kX''+R(X',X''))\Bz_2 =0\Big)^4\\
			\le & \P_{X',X'',Y''\dots}\Big( [A(X',X'') -A(X',Y'')]\Bz_1 + ( X''Z_1\dots Z_kX'' - Y''Z_1\dots Z_kY'' +R(X',X'') - R(X',Y''))\Bz_2 =0\Big)^2 \\
			\le & \P_{X',Y', X'',Y''\dots}\Big( [A(X',X'') -A(X',Y'')-A(Y', X'') + A(Y',Y'')]\Bz_1 \\
			& \ \  \  \  \  \  \  \  \  \ \  \ \  \  \  \  \  \  \  \  \ \  \ \  \  \  \  \  \  \  \  \ \   + (R(X',X'') - R(X',Y'')-R(Y',X'') + R(Y',Y''))\Bz_2 =0\Big).
			\end{align*}
			It is crucial to note that by doing so, if the highest degree of $X$ (which was decomposed into $X'+X''$) was also $k$, then the products having $k$ factors of $X'$ (and also $Y',X'', Y''$) are canceled out in $A(X',X'') -A(X',Y'')-A(Y', X'') + A(Y',Y'')$. 
			\vskip .1in
			\item In summary, after each round of the decoupling process, we will not create higher polynomials elsewhere but replace $X$ which appears $k$ times in the product form by four matrices $X',X'', Y',Y''$ which appear at most $k-1$ times. Hence, after $4^d$ steps of iterating the process, we will create a sum of many multilinear forms in which each matrix factor appears at most once. In other words they might have the form
			$$\P(Q^{d}\Bv_1 + \sum_{j=0}^{d-1} t_j Q^j \Bg  = \Bv_1)^{4^d} \le \P((\sum_{\sigma} X_{\sigma(1)}X_{\sigma(2)}\dots X_{\sigma(d)} + R)\Bz_1 + S \Bz_2=0),$$
			where $R,S$ are also multilinear forms which might also contain $X_1,\dots, X_d, \dots$ but the total degrees are smaller than $d$.  
		\end{itemize}
		
		Notice that in the matrix $X_i$ the are approximately $n/4^d$ non-zero rows.
		\begin{proof}(of Theorem \ref{thm:char'}) We have seen that 
			\begin{equation}\label{eqn:multilinear}\P\Big(Q^{d}\Bv_1 + \sum_{j=0}^{d-1} t_j Q^j \Bg  = \Bv_1\Big)^{4^d} \le \P\Big((X_{1}X_{2}\dots X_{d} \Bv_1 + R \Bv_1 + S \Bg =0\Big).
			\end{equation}
			Set 
			$$n'=n/4^d.$$

			{\bf A simplified case.} In order to motivate our next step,  let us assume for now that our RHS \eqref{eqn:multilinear} has only the term $X_1 \dots X_d$, that we are estimating the probability that  $\P((X_{1}X_{2}\dots X_{d} \Bv_1=0)$.

			Let $1\le m_0 \le m$ be fixed. We will assume $\Bv_1$ to have exactly $m_0$ non-zero entries (noting that taking union bound over $m_0$ will not significantly change our bounds), then by the fact that $\sup_a \P(\sum_{i=1}^m \xi_i x_i=a) \le 1-c$ for some absolute constant $c$ and by the fact that $X_d$ has $C_d n$ iid row vectors (where $C_d \approx 1/4d$) we have that 
			\begin{equation}\label{eqn:Xdv1}
			\P(X_d \Bv_1 \mbox{ has at least $C_d n$ non-zero entries})\ge 1- c^{C_d n}.
			\end{equation}
			In the next step, conditioning on this event of $X_d$, we consider the vector $X_{d-1}(X_d \Bv_1)$ and apply the following fact (by relying on Theorem \ref{theorem:LO}) 
			\begin{claim}\label{claim:w:sparse} Assume that $\Bw$ is not $C_dn$-sparse, and $X$ is a random matrix with $n'$ iid rows as in $M$. Then for sufficiently large $p$, and for $p\lesssim \sqrt{n}$ 
				$$\P(X\Bw \mbox{ is $n'/3$-sparse}) \le p^{-n'/2}.$$  
			\end{claim}
			
			\begin{proof} Note that if $\Bw$ is not $C_d n$-sparse then by Theorem \ref{theorem:LO}
				\begin{equation}\label{eqn:simple:f}
				\sup_{f_i \in \F_p}\P(\sum_i x_i w_i =f_i) \le 1/p + 1/\sqrt{C_d n} \lesssim 1/p.
				\end{equation}
				As such, as $p$ is sufficiently large the probability under consideration is bounded by 
				$$(C/p)^{(1-1/3)n'} \binom{n'}{n'/3} \le (C/p)^{n'/2},$$
				where we take a union bound over all possible positions for the zero coordinates of $X \Bw$. 
			\end{proof}
			We remark the the above might continue to hold for small $p$, but we will not be focusing on this case for simplicity.
			
			We next iterate Claim \ref{claim:w:sparse}, by taking union bound and with an assumption that $m\lesssim_d \sqrt{n}$ we obtain a bound 
			$$p^m (2/p)^{c_d'n} \le (1/p)^{c_d'' n}$$ 
			for the event that $X_1\dots X_d \Bv_1=0$ (or $X_1\dots X_d \Bv_1$ is $m$-sparse). 
			
			{\bf General case.} First let us record here a slightly more general variant of Claim \ref{claim:w:sparse}. We say that a random vector $\Bw$ is {\it $m$-free} if all but at most $m$ coordinates of $\Bw$ are determined. So for instance $m$-sparse vectors are $m$-free because all but at most $m$ coordinates of $\Bw$ are zero. By an identical proof to Claim \ref{claim:w:sparse}, we have
			
			\begin{claim}\label{claim:w:sparse'} Assume that $\Bw$ is not $C_dn$-free, and $X$ is a random matrix with $n'$ iid rows as in $M$. Then for sufficiently large $p$, and for $p\lesssim \sqrt{n}$ 
				$$\P(X\Bw \mbox{ is $n'/3$-free}) \le p^{-n'/2}.$$  
			\end{claim}
			To continue, recall that we would like to estimate $\P(\sum_{\sigma} X_{\sigma(1)}X_{\sigma(2)}\dots X_{\sigma(d)} + R)\Bz_1 + S \Bz_2=0)$, and here we cannot expose $X_d$, and then $X_{d-1}$, etc one by one in order as in the simplified case above because $X_i$ might not appear exactly in the $i$-th position of each multilinear forms. However we can adjust the process using the following observation.
			\begin{fact}\label{fact:zero} For any $i$, and for any vector $\Bu$, the vector $X_i \Bu$ of $\R^n$ have non-zero entries only in the $i$-th block. In particularly, if a vector $\sum_i X_i\Bu_i$ is $n'$-sparse, then $X_1 \Bu_1$ is also $n'$-sparse.  
			\end{fact}
			Now we describe the method. First, basing on Fact \ref{fact:zero} we just need to address all the multilinear forms beginning with $X_1$. The vector obtained by summing over these forms is of type $X_1 (\sum_{i=2}^{d} a_i X_i  R_i \Bw_i)$. To estimate the probability that this vector is not $m$-free, we will focus only on the submatrix restricted by the columns of $X_1$ having the same index as the rows of $X_2$, and condition on the remaining entries of $X_1$. Only using the randomness from this submatrix $X_1'$ of $X_1$, it suffices to show that $(X_1'+F) (X_2 R_2 \Bv)$ is not $n'$-free with high probability. To show this, assume that we already know that $X_2R_2 \Bv$ is non-sparse, we then can use \eqref{eqn:simple:f} (i.e. Claim \ref{claim:w:sparse'}), where $f_i$ is allowed to depend on $F$. Now to show that $X_2R_2 \Bv$ is not $n'$-free with high probability (where the randomness is on $X_2,\dots, X_d$), where $R_2$ is a sum of multilinear forms without $X_1$ and $X_2$, we again focus on the submatrix $X_2'$ of $X_2$ restricted by the columns having the same indices as that of $X_3$, and continue forward. So in the last step of our argument (or the first step if we go backward as in the simplified case above), we just need to show that $X_d\Bv_1$ is not $n'$-free with high probability with respect to the randomness of $X_d$ and for some appropriate number $n' = \Theta_d(n)$, but this was exactly \eqref{eqn:Xdv1}.
		\end{proof}

		\section{Proof of Theorem \ref{thm:free:universal}}\label{section:free} 
		
		We are going to prove the following result.
		
		\begin{proposition}[asymptotically independence]\label{prop:free:d} Let $d\ge 1$ be fixed. Then for any distinct numbers $a_1,\dots,a_d$ in $\F_p$, 
			$$\P(\CE_{a_d}\wedge \dots \wedge \CE_{a_d}) = \frac{1}{p^d} + \exp(-c_d n/p^2),$$
			where $\CE_a$ is the event that $a$ is an eigenvalue of $M$.
		\end{proposition}

		\begin{proof}(of Theorem \ref{thm:free:universal}) Let $X_{n,p}$ denote the random variable
			$$X_{n,p} = \sum_{a\in \F_p} 1_{\CE_a}.$$
			By Proposition \ref{prop:free:d}, one easily has $\E X_{n,p} = 1 + p \times \exp(-c_1 n/p^2) = 1+o(1)$, and more generally for any fixed integer $d\ge 1$
			$$\lim_{n \to \infty} \E X_{n,p}^d = \E X^d,$$
			where $X \sim Pois(1)$. It thus follows that $X_{n,p}$ is asymptotically distributed as $Pois(1)$. In particularly 
			$$\lim_{n,p \to \infty} \P(X_{n,p}=0) = \P(X=0) = e^{-1}.$$  
		\end{proof}

		It remains to justify the above independence lemma.

		\subsection{Proof of Proposition \ref{prop:free:d}} The event $\wedge \CE_{a_i}$ is equivalent to the event that there exist $\Bv_i, 1\le i\le d$ such that 
		$$M\Bv_1 =a_1 \Bv_1 \wedge \dots \wedge M\Bv_d=a_d \Bv_d.$$
		We condition on $M_{[n] \times [n-1]}$. Let $\Bu_i$ be a normal vector of $M_{[n] \times [n-1]}-a_i$. The event  $\wedge \CE_{a_i}$ then implies that 
		\begin{equation}\label{eqn:Eai}
		(X_n -a_i \Be_n) \cdot \Bu_i = 0, 1\le i\le d.
		\end{equation}
		We are going to show that the probability of this event with the randomness on $X_n$ is $\frac{1}{p^d}$ modulo a small error term for almost all realization of $M_{[n] \times [n-1]}$. To be more precise, we will restrict on the following event of probability close to 1.

		\begin{lemma}\label{lemma:free:sparse} Let  $c_d$ be a sufficiently small constant to be chosen later. With probability at least $1-O(\exp(-\Theta(n))$ with respect to $M_{[n] \times [n-1]}$ , none of the vectors $ \sum_u \be_i \Bu_i$ for all $\beta_i \in \F_p$, not  all zero, is $c_d n$-sparse.
		\end{lemma}
		Assuming this for now we can conclude our main result.
		\begin{proof}(of  Proposition \ref{prop:free:d}) Conditioning on the event consider in Lemma \ref{lemma:free:sparse}, we can then just follow the proof of Theorem \ref{thm:char'} applied to the events in \eqref{eqn:Eai}. More precisely, the estimates following Equation \eqref{eqn:Yg} hold with $g_i= a_i \Be_n \cdot \Bu_i$ for $1\le i\le d$.
		\end{proof}

		\subsection{Proof of Lemma  \ref{lemma:free:sparse}} Set $m=c_d n$, for sufficiently small $c_d$. By taking union bound (with a loss of a multiplicative factor $p^d$ in probability), it suffices to consider the probability that $\Bw = \sum_u \be_i \Bu_i$ is $m$-sparse for a fixed choice of $\beta_1,\dots, \beta_d$.
		
		\begin{lemma}\label{lemma:free:sparse:beta} With probability at least $1-O(\exp(-\Theta(n)))$, $\Bw$ cannot be $m$-sparse.
		\end{lemma}
		As in the previous section, let $Q$ denote the matrix $M_{[n-1] \times [n-1]}^T$. Our result says that the vectors $\Bw$ can be viewed as null vector of a polynomial of  degree $d$ of $Q$. More specifically, let $X$ be the last row of $M_{[n] \times [n-1]}$. 
		\begin{claim}\label{claim:free:poly}
			There exist coefficients $\a_i$ and $\a_i'$ such that 
			$$\sum_{i=0}^d \a_i Q^i \Bw + \sum_{i=0}^{d-1} \a_i' Q^i X=0.$$
		\end{claim}
		It is clear that, by using this claim, we can complete the proof of Lemma \ref{lemma:free:sparse} by following exactly the decoupling process in the proof of  Theorem \ref{thm:char'} in the previous section (which in turns yield the bound $1-\exp(-\Theta(n))$, obtained at the last step of the process). This bound is clearly strong enough to absorb all union bounds of type $p^{O(d)}$ given the range of $p$. It thus remains to justify the result above.
		
		\begin{proof}(of Claim \ref{claim:free:poly}) By fixing $f_i = u_{in}$ (and hence we lose another multiplicative factor of $p^d$ in probability), and $\Bw_i$ is the concatination of $\Bu_i$, by $(M-a_i)\Bu_i=0$ we have 
			$$(Q - a_i) \Bw_i = f_i X.$$
			As such
			$$(Q-a_1) \sum_i \beta_i \Bw_i = \beta_1 f_1 X + \sum_{i\ge 2} \beta_i f_i X +  \sum_{i\ge 2} \beta_i (\a_i -\a_1) \Bw_i.$$
			In other words, under the action by $Q-a_1$, we eliminate $\Bw_1$ (or we changed it to a deterministic vector). Iterating the process, we then obtain that 
			$$\sum_{i=0}^d \a_i Q^i \Bw + \sum_{i=0}^{d-1} \a_i' Q^i X=0,$$
			where $\a_d \neq 0, \a_i, \a_i'$ depend on $\beta_1,\dots, \beta_d, a_1,\dots, a_d$ and $X$. 
		\end{proof}

		\section{Remarks}\label{section:remarks}
		First, among the three characterizations provided, Theorem \ref{theorem:ILO} was less effective in our current applications because of the polynomial restriction, but this result is expected to have other implications beyond random matrix theory because of its near optimality. The remaining two characterizations, Theorem \ref{theorem:inverse:modp} and Theorem \ref{theorem:Halasz}, yield sub-exponential bounds in application. The later is more amenable to perturbations (that we don't have to assume $\|M\|_2=O(\sqrt{n})$).		
		
		Second, we remark that the error bounds in Theorem \ref{thm:char:universal'} and in Proposition \ref{prop:free:d} are of the form $\exp(-n^{1-o(1)}/p^2)$, which were obtained by applying Corollary \ref{cor:smallp}. Compared to the justification for the uniform model in Section \ref{section:uniform}, our approach seems to be natural and does explain the main terms in Theorem \ref{thm:char:universal} (obtained from Theorem \ref{thm:char:universal'} and Theorem \ref{thm:char''}) and in Theorem \ref{thm:free:universal} (obtained from Proposition \ref{prop:free:d}). Following our treatment of Section \ref{section:rank}, it is natural to expect that these error bounds can be made $\exp(-n^c)$, but for this improvement one has to show that the vectors $\Bv_1$ in \eqref{eqn:Q} or $\Bw$ in Claim \ref{claim:free:poly} to have large ULCD or large $R_k$, but this task seems to be extremely challenging.  
		
		Universality is an extremely complicated phenomenon. While we have addressed only a few universal examples for random matrices in (prime) fields in the current note, there remains so many interesting and tantalizing questions. Beside the obvious (and doable) direction of extending the current results to general finite fields $\F_q$, we conjecture that the following statistics of the uniform model are universal in terms of iid random matrix model:
		
		\begin{itemize}
			\item the distributions of $\la_{\phi_1}, \dots, \la_{\phi_k}$ are asymptotically independent for different irreducible polynomials $\phi_1,\dots, \phi_k$ (which would then generalize \eqref{eqn:independence});
			\vskip .1in
			\item the results of Stong and of Hansen and Schmutz \cite{HSch} connecting the distribution of degrees of irreducible factors of the characteristic polynomial to the cycle lengths of a random permutation.
		\end{itemize}
		Lastly, we conjecture that for a fixed random matrix model (such as the Bernoulli $(0,1)$ or $(-1,1)$ model), the considered statistics over $\F_p$ for different primes $p$ are asymptotically independent.


		
		
		{{ \bf Acknowledgements.} The authors are thankful to  J. Koenig for helpful comments. The first author is supported in part by the National Science Foundation postdoctoral fellowship DMS-1702533.  The last two authors are partially supported by National Science Foundation grants DMS-1600782 and CAREER DMS-1752345.}

		\appendix
		
		\section{Tensorization lemma}\label{appendix:tensor}
		The following is an analog of \cite[Lemma 2.2]{RV}.
		\begin{lemma}\label{lemma:tensor} Let $K,\delta_0\ge 0$ be given.
			Assume that $\xi_1,\dots, \xi_n$ are iid real-valued random variable and that $\P(\|\xi_i\|_{\R/\Z}< \delta )\le K\delta$ for all $\delta\ge \delta_0$. Then 
			$$\P(\|\xi_1\|_{\R/\Z}^2+\dots+\|\xi_n\|_{\R/\Z}^2  <\delta n)\le (C_0K\delta)^n,$$
			where $C_0$ is absolute.
		\end{lemma}
		
		\begin{proof} Assume that $\delta\ge \delta_0$. By Chebyshev's inequality
			$$\P(\|\xi_1\|_{\R/\Z}^2+\dots+\|\xi_n\|_{\R/\Z}^2   \le \delta n) \le \E \exp( n- \sum_{i=1}^n {\|\xi_i\|_{\R/\Z}}^2/\delta)=\exp(n) \prod_{i=1}^n \E \exp(-{\|\xi_i\|_{\R/\Z}}^2/\delta).$$
			On the other hand,
			$$ \E \exp(-{\|\xi_i\|_{\R/\Z}}^2/\delta) = \int_0^1 \P(\exp(-{\|\xi_i\|_{\R/\Z}}^2/\delta)>s) ds = \int_0^\infty 2u \exp(-u^2) \P(\|\xi_i\|_{\R/\Z}<\delta u)du.$$
			For $0\le u\le 1$ we use $\P(\|\xi_i\|_{\R/\Z} \le \delta u)\le \P(\|\xi_i\|_{\R/\Z}\le \delta) \le K \delta$, while for $u\ge 1$ we have $\P(\|\xi_i\|_{\R/\Z} \le \delta u)\le K \delta u$. Thus
			$$ \E \exp(-{\|\xi_i\|_{\R/\Z}}^2/\delta)  =  \int_0^1 2u \exp(-u^2) K \delta  du +  \int_1^\infty 2u \exp(-u^2) K \delta u du \le C_0 K \delta,$$ as desired.
		\end{proof}

	\end{document}